\documentclass{amsart}

\usepackage{amsfonts}
\usepackage{amsmath,amsthm,amssymb,amssymb}
\usepackage{cite}
\usepackage{mathtools}
\usepackage{url}
\usepackage{tikz-cd}

\usepackage{hyperref}

\usepackage{amscd}
\usepackage{mathrsfs}
\usepackage{changepage}
\usepackage{tipa}
\usepackage{graphics}
\usepackage{times}
\usepackage{extarrows}
\usepackage[export]{adjustbox}
\usepackage[all,cmtip]{xy}

\theoremstyle{remark}
\numberwithin{equation}{section}
\theoremstyle{plain}
\theoremstyle{definition}
\theoremstyle{remark}

\newtheorem{thm}{Theorem}[section]
\newtheorem{lem}{Lemma}[section]

\newtheorem{prop}{Proposition}[section]
\newtheorem{rem}{Remark}[section]
\newtheorem{defn}{Definition}[section]

\newtheorem{ex}{Example}[section]
\def\={\overset{\text{def}}{=}}

\def\||{\parallel}

\usepackage{lipsum}

\begin{document}

\bibliographystyle{plain}

\title{Relative Equivariant Coarse Index Theorem and Relative $L^2$-Index Theorem}

\author{Xiaoman Chen}
\author{Yanlin Liu}
\author{Dapeng Zhou} 

\address[Xiaoman Chen]{School of mathematical Sciences, Fudan University.}
\email{\textbf{xchen@fudan.edu.cn}}

\address[Yanlin Liu]{School of mathematical Sciences, Fudan University.}
\email{\textbf{16110180006@fudan.edu.cn}} 

\address[Dapeng Zhou]{Research center for operator algebras, East China Normal University.}
\email{\textbf{giantroczhou@126.com}}

\date{}
\maketitle

\begin{abstract}In this paper, we give a definition of the relative equivariant coarse index for proper actions and derive a relative equivariant coarse index theorem connecting this index with the localized equivariant coarse indices. This is an equivariant version of Roe's relative coarse index theorem in \cite{Roe2016}. Furthermore, we present a definition of the relative $L^2$-index and prove a relative $L^2$-index theorem which is a relative version of Atiyah's $L^2$-index theorem in \cite{Atiyah1992}.
\end{abstract}	

\section{Introduction}

The main aim of this paper is to define the relative equivariant coarse index for proper actions, prove a relative equivariant coarse index theorem and a relative $L^2$-index theorem.

The original version of the relative index theorem was introduced by Gromov and Lawson \cite{Gromov} as a new analytic tool to research on positive scalar curvature question. Their relative index theorem can be expressed as follows. Suppose that $M_1$ and $M_2$ are complete Riemannian manifolds equipped with generalized Dirac operators $D_1$ and $D_2$ respectively, acting on Dirac bundles $S_1$ and $S_2$. Suppose further that there is an isometry of manifolds and bundles which identifies $D_1$ with $D_2$ outside of compact subsets $Y_1\subseteq M_1$ and $Y_2\subseteq M_2$. They obtained compact manifolds $\widetilde{M_i}$ with elliptic operators $\widetilde{D_i}$ by compactifying  each of the $M_i$ identically outside $Y_i$ and then define the relative index as
\[
\text{Index}_r(D_1,D_2)=\text{Index}(\widetilde{D_1})-\text{Index}(\widetilde{D_2})\in \mathbb{Z}
\]
where Index$(\widetilde{D_i})$ are ordinary Fredholm indices. Their relative index theorem states
\begin{thm}(Gromov and Lawson \cite{Gromov})
Based on above conditions, if $D_1$ and $D_2$ have uniformly positive Weitzenbock curvature operators at infinity, then
\[
\text{Index}_r(D_1,D_2)=\text{Index}(D_1)-\text{Index}(D_2).
\]
\end{thm}

Bunke \cite{Bunke1992} generalized this relative index theorem and developed the framework for the heat equation method in the relative index theory. The paper \cite{Bunke1992} explains the relations between the relative index theory of Gromov and Lawson \cite{Gromov} and the supersymmetric scattering theory introduced by Borisov, M\"{u}ller and Schrader \cite{Borisov1988}. Bunke \cite{Bunke1995} gave a $K$-theoretic variant of the relative index theorem and applied this theorem to the positive scalar curvature problem.

Roe \cite{Roe1991} unified the relative index constructions of Gromov and Lawson \cite{Gromov}, Borisov, M\"{u}ller and Schrader \cite{Borisov1988}, and Julg \cite{Julg1988} in the context of the operator algebras developed by himself \cite{Roe1990}. Moreover, he \cite{Roe1991} gave a heat equation proof of the relative index theorem. Furthermore, Roe \cite{Roe2016} extended his definition of the relative index such that the subsets $Y_1$ and $Y_2$ outside which the operators $D_1$ and $D_2$ coincide are only assumed to be closed rather than compact, then he also provided a generalization of the relative index theorem in this case.

Xie and Yu \cite{XieYu2014} proved a general relative higher index theorem for complete manifolds with positive scalar curvature towards infinity. They made use of this theorem to consider Riemannian metrics of positive scalar curvature on manifolds.

Donnelly \cite{Donnelly1987} derived a relative signature theorem for manifolds with strictly positive essential spectrum. This is analogous to the result of Gromov and Lawson \cite{Gromov}.

Anghel's proof in \cite{Anghel1990} was based on Gromov-Lawson's relative index theorem \cite{Gromov}. He \cite{Anghel1993} also made use of Gromov-Lawson's relative index theorem to reduce the problem from arbitrary manifolds to Riemannian products.

Nazaikinskii \cite{Nazaikinskii2015Relative} proved an analogue of Gromov-Lawson's relative index theorems for $K$-homology classes. He \cite{Nazaikinskii2015On} gave a generalization of Bunke's relative index theorem in \cite{Bunke1995} to the $KK$-theory setting and proved a relative index type theorem that compares certain elements of the Kasparov $KK$-group. 

All of these results show that the relative index theory has been developed in numerous different cases and the relative index method has found many impressive applications towards geometry and topology, especially to those related to the positive scalar curvature problem and the Novikov conjecture.

Up to now, there is little investigation concerning the relative index theory for Dirac operators on manifolds equipped with group actions, except for Xie and Yu's work in \cite{XieYu2014}. In this paper, we study the equivariant version of the relative index where the isometric group actions are proper. This is an extension for Gromov-Lawson's and Roe's relative index theorey. The main method we used is Roe's idea in \cite{Roe2016}. Unlike Xie and Yu's relative index in \cite{XieYu2014} where the two Dirac operators $D_1$ and $D_2$ need to coincide outside two compact subsets $Y_1$ and $Y_2$, in this paper,  we only require the subsets $Y_1$ and $Y_2$ to be closed. Unlike Roe's conditions \cite{Roe2016} that $Y_1$ is coarse equivalent to $Y_2$ and $M_1$ is coarse equivalent to $M_2$, we only need the existence of an equivariant coarse map from $Y_2$ to $Y_1$ and $M_2$ to $M_1$.

To define our relative equivariant coarse index, we summarize the following conditions as a set of relative equivariant coarse index data.
\begin{enumerate}
  \item[(i)] A countable discrete group $G$ acts on complete Riemannian manifolds $M_1$ and $M_2$ properly by isometries. The manifolds $M_i$ are equipped with Dirac bundles $S_i$ on which the Dirac operators $D_i$ act, and $Y_i\subseteq M_i$ are $G$-invariant closed subsets. The group $G$ also acts on $S_i$ by isometries and the projections $S_i\rightarrow M_i$ are $G$-equivariant. 
  \item[(ii)]A $G$-equivariant isometric diffeomorphism $h:M_1-Y_1\rightarrow M_2-Y_2$ covered by bundle isomorphism $\bar{h}$ satisfies
\[
\bar{h}\circ(D_1\varphi)\circ h^{-1}=D_2(\bar{h}\circ \varphi \circ h^{-1}),\,\forall \varphi\in L^2(M_1-Y_1,S_1).
\]
\item[(iii)]A $G$-equivariant coarse map $q:Y_2\longrightarrow Y_1$ satisfies that the map $M_2\longrightarrow M_1$ defined by $h^{-1}$ on $M_2-Y_2$ and $q$ on $Y_2$ is still a $G$-equivariant coarse map.
\end{enumerate}

Based on above conditions, we obtain the relative equivariant coarse index
\[\text{Index}_r^G(D_1,D_2)\in K_m(C^*(Y_1)^G)\]
where $m=\dim M_1=\dim M_2$. Our relative equivariant coarse index theorem states
\begin{thm}In the circumstances described above, if the curvature operators $\mathcal{R}_i$ on the Dirac bundles $S_i$ are uniformly positive outside $Z_i=\overline{O(Y_i,s)}$ for some $s>0$, namely, there is an $\varepsilon>0$ such that
\[
\mathcal{R}_{ix}\geq \varepsilon^2I,\,\forall x\in M_i-Z_i,
\]
then $D_i$ have the localized equivariant coarse indices
 \[
 \text{Index}_{Z_i}^G D_i \in K_m(C^*_{Z_i}(M)^G)\overset{\cong}{\longrightarrow} K_m(C^*(Y_i)^G) \ni\text{Index}_{Y_i}^G D_i
 \]
and the following identity holds in $K_m(C^*(Y_1)^G)$
\[
\text{Index}^G_r(D_1,D_2)=\text{Index}_{Y_1}^G D_1-q_*\Big(\text{Index}_{Y_2}^G D_2\Big)
\]
where $q_*:K_*(C^*(Y_2)^G)\longrightarrow K_*(C^*(Y_1)^G)$ is induced from the coarse map $q$.
\end{thm}

On the other hand, the $L^2$-index theorem is due to Atiyah \cite{Atiyah1992} which expresses the index of an Dirac operator on a closed Riemannian manifold $M$ in terms of the $G$-equivariant index of some regular covering $\pi^\#:M^\#\longrightarrow M$, with $G$ the group of covering transformations. In this paper, we give a general definition of relative $L^2$-index as follows. For a set of relative equivariant coarse index data $(M_i,S_i,D_i,Y_i,h,q)$ over $G$ where $M_i$ are even dimensional and $G$ acts on $Y_i$ cocompactly, the relative $L^2$-index is defined by
 \[
 L^2\text{-Index}_r^G(D_1,D_2)=\left(\text{Tr}_{Y_1}\right)_*\left(\text{Index}_r^G(D_1,D_2)\right)\in \mathbb{R}.
 \]
Here $\text{Tr}_{Y_1}$ is the canonical trace on $C^*(Y_1)^G$.

 By introducing a densely defined, lower semicontinuous positive trace on the localized equivariant Roe algebras and proving some related properties, we obtain a  formula computing the relative $L^2$-index when $G$ acts on $M_i$ freely. With the help of this formula, we prove the following relative $L^2$-index theorem.
\begin{thm}Let $(M_i,S_i,D_i,Y_i,h,q)$ be a set of relative coarse index data with even dimensions where $Y_i$ are compact subsets and $(M_i^\#,S_i^\#,D_i^\#,Y_i^\#,h^\#,q^\#)$ be a set of relative equivariant coarse index data over $G$ acting on $M_i^\#$ freely and on $Y_i^\#$ cocompactly. If the maps $\pi_i:M_i^\#\longrightarrow M_i$ are regular $G$-covers and $M_i^\#$ are equipped with the pullback Riemannian metric, Dirac bundles $S^\#$ and Dirac operators $D_i^\#$ by $\pi_i$ where the subsets $Y_i^\#$ are the preimage of $Y_i$, then
\[%
L^2\text{-Index}^G_r\left(D_1^\#,D_2^\#\right)=L^2\text{-Index}_r(D_1,D_2)\in \mathbb{Z}.
\]
\end{thm}

This theorem can be viewed as a relative version of Atiyah's result in \cite{Atiyah1992}.

This paper is organized in the following way. In section 2, we present the background knowledge for the next sections. In section 3, we define our relative equivariant coarse index with more detailed explanation, such as alternative ways of approaching the definition and the fact that the resulting index does not depend on various choices in the process. In section 4, the relative equivariant coarse index theorem is proved. After introducing densely defined, lower seimcontinuous positive traces on localized equivariant Roe algebras, we define the relative $L^2$-index in section 5 and obtain a formula for the relative $L^2$-index for computing it. In section 6, we prove the relative $L^2$-index theorem.
\section{Preliminaries}

In this section, we introduce the background knowledge and the basic tools that will be used in the following sections of this paper.

\subsection{Geometric Modules and Associated $C^*$-Algebras}

\text{ }

In this subsection, we recall the definitions in coarse geometry and associated $C^*$-algebras which is one of the main tools to study index theory. Related theories can be found in \cite{Higson2000} and \cite{Rufus}.

Assume that $G$ is a countable discrete group acting on proper metric spaces $X$ and $X'$ properly by isometries. Proper metric space means every closed bounded subset in the metric space is compact, and proper actions signify $\{g\in G:gK\cap K\neq\emptyset\}$ is finite for every compact $K\subseteq X$. Moreover, the $G$-action on $X$ is called compact if $X/G$ is compact. We call a Borel subset $E\subseteq X$ a fundamental domain if $X=\coprod_{g\in G}gE.$ 

\begin{defn}(cf. \cite{Rufus})
  A triple $(\rho,H,U)$ is called a geometric $X$-$G$ module if $H$ is a separable Hilbert space, $\rho:C_0(X)\longrightarrow B(H)$ is a nondegenerate *-homomorphism mapping nonzero elements to noncompact operators and $U:G\longrightarrow U(H)$ is a group homomorphism satisfying
  \[\rho(f\circ g^{-1})=U_g \rho(f) U_g^*,\ \  \forall f\in C_0(X), g\in G.\]
\end{defn}

Remark that for every geometric $X$-$G$ module there is a unique *-homomorphism extension for $\rho$ on $B(X)$ such that if bounded sequence $\{f_n\}$ in $B(X)$ converge to $f$ pointwisely then $f$ is in $B(X)$ and $\{\rho(f_n)\}$ converge to $\rho(f)$ strongly, where $B(X)$ is the set of all bounded complex-valued Borel measurable functions on $X$. For $T\in B(H)$ and $f\in B(X)$, sometimes we denote $\rho(f)T$ and $T\rho(f)$ by $fT$ and $Tf$ respectively for simplicity. We may also call $H$ a geometric $X$-$G$ module.

\begin{defn}(cf. \cite{Rufus})
  The geometric $X$-$G$ module $(\rho,H,U)$ is called locally free if for every finite subgroup $F<G$ and every $F$-invariant Borel subset $E\subseteq X$, there exists a Hilbert space $H'$ and a unitary operator $V:\chi_E H\longrightarrow l^2(F)\otimes H'$ satisfying the commutative diagram
        \[\xymatrix{
          \chi_E H \ar[d]_{U_f} \ar[r]^{V\ \ \ \ \ } & l^2(F)\otimes H' \ar[d]^{\lambda_f\otimes 1} \\
          \chi_E H  \ar[r]^{V\ \ \ \ \ } & l^2(F)\otimes H'   }\]
         for every $f\in F$. Here $\lambda_f$ is the left regular action on $l^2(F)$.
\end{defn}
\begin{prop}\label{locallyfree}
If $(\rho,H,U)$ is a geometric $X$-$G$ module, then there is a locally free geometric $X$-$G$ module $(\rho\otimes 1, H\otimes l^2(G), U\otimes\lambda)$ defined by
\[
\rho\otimes 1 :C_0(X)\longrightarrow B(H\otimes l^2(G)),f\longmapsto \rho(f)\otimes 1,
\]
\[
U\otimes \lambda:G\longrightarrow U(H\otimes l^2(G)), g\longmapsto U_g\otimes \lambda_g,
\]
where $\lambda$ is the left regular representation of $G$ on $l^2(G)$.
\end{prop}

\begin{proof}
It can be check that $(\rho\otimes 1, H\otimes l^2(G), U\otimes\lambda)$ is a geometric $X$-$G$ module since
\[
(U_g\otimes \lambda_g)(\rho(f)\otimes 1)(U_g\otimes \lambda_g)^*=(U_g\rho(f)U_g^*)\otimes 1=\rho(f\circ g^{-1})\otimes 1,\forall f\in C_0(X),g\in G.
\]

For every finite subgroup $F$ in $G$ and every $F$-invariant Borel subset $E$ in $X$, there is a subset $\Gamma\subseteq G$ such that
\[
G=\coprod_{a\in \Gamma} Fa.
\]
Therefore,
\[
G/F:=\{Fa:a\in G\}=\{Fa:a\in \Gamma\},
\]
\[
\forall g\in G,\ \exists !g'\in F,\ g''\in \Gamma,\ s.t. \ \ g=g'g''.
\]
Define the unitary operator $V$ as
\[
V: \chi_E H\otimes l^2(G)\longrightarrow l^2(F)\otimes \chi_E H\otimes l^2(G/F),
\]
\[
\varphi\otimes \delta_g\longmapsto \delta_{g'}\otimes (\varphi\circ g)\otimes \delta_{Fg''},\ \forall \varphi\in \chi_E H, \forall g\in G.
\]
Here $\delta_g$ is the characteristic function on $\{g\}$. Then for every $f\in F$, the following diagram is commutative
\[
\xymatrix{
  \chi_E H\otimes l^2(G) \ar[d]_{U_f\otimes \lambda_f} \ar[r]^{V\ \ \ \ \ \ \ \ } & l^2(F)\otimes \chi_E H\otimes l^2(G/F) \ar[d]^{\lambda_f\otimes 1\otimes 1} \\
  \chi_E H\otimes l^2(G) \ar[r]^{V\ \ \ \ \ \ \ \ } & l^2(F)\otimes \chi_E H\otimes l^2(G/F)   }
\]
since
\[
\xymatrix{
  \varphi\otimes \delta_g \ar[d]_{U_f\otimes \lambda_f} \ar[r]^{V\ \ \ \ \ \ } & \delta_{g'}\otimes (\varphi\circ g)\otimes \delta_{Fg''} \ar[d]^{\lambda_f\otimes 1\otimes 1} \\
  (\varphi\circ {f^{-1}})\otimes\delta_{fg} \ar[r]^{V\ \ \ \ \ \ } & \delta_{fg'}\otimes(\varphi\circ g)\otimes \delta_{Fg''}.}
\]
This shows that $(\rho\otimes 1, H\otimes l^2(G), U\otimes\lambda)$ is locally free.
\end{proof}

\begin{defn}(cf. \cite{Higson2000}, \cite{Rufus}, \cite{Yu1997})
For geometric $X$-$G$ module $(\rho,H,U)$, geometric $X'$-$G$ module $(\rho',H',U')$, an operator $T\in B(H,H')$ and a map $f:X\rightarrow X'$,

\begin{enumerate}
  \item $f$ is equivariant if it is commutative with the group actions on $X$ and $X'$;
  \item $f$ is called a coarse map if the preimage of any bounded subset of $X'$ is still bounded in $X$ and
  \[
  \forall R,\exists r,\forall x,y\in X,d(x,y)<R\Longrightarrow d(f(x),f(y))<r;
  \]
  \item $f$ is called a uniform map if it is coarse and continuous;
  \item $f$ is called close to another map $f':X\longrightarrow X$ if there is $\varepsilon >0$ such that $d(f(x),f'(x))<\varepsilon$ for every $x\in X$;
  \item $f$ is called a coarse equivalence if it is a coarse map with dense image in $X'$; in this case, $X$ is called coarse equivalent to $X'$;
  \item $T$ is equivariant if $TU_g=U_gT$ for every $g\in G$;
  \item the support of $T$ is defined by
  \[
  \text{supp}  (T)=\{(x',x)\in X'\times X:\forall \text{ open } A\ni x', B\ni x \Longrightarrow \chi_A T\chi_B\neq 0\};
  \]
  \item isometric operator $T$ is called an $\varepsilon$-cover of the coarse map $f:X\rightarrow X'$ if
    \[
    \forall (x',x)\in \text{supp}  (T)\Longrightarrow d(x',f(x))\leq\varepsilon;
    \]
  \item isometric operator $T$ is called a uniform $\varepsilon$-cover of the uniform map $f:X\rightarrow X'$ if $T$ is an $\varepsilon$-cover of $f$ and
      \[
      T^*\rho'(\phi)T-\rho(\phi\circ f)
      \]
      is an compact operator on $H$ for every $\phi\in C_0(X')$.
\end{enumerate}
\end{defn}
\begin{defn}(cf. \cite{Rufus}, \cite{XieandYu14positivescalarcurvature})
For a geometric $X$-$G$ module $(\rho,H,U)$ and an operator $T\in B(H)$,
\begin{enumerate}
  \item denote by $B(H)^G$ the set of all equivariant bounded linear operators on $H$;
  \item the propagation of $T$ is given by
  \[
  \text{prop}(T)=\sup\{d(x',x):(x',x)\in \text{supp}(T)
  \}\in [0,\infty];
  \]
  \item $T$ is locally compact if $fT$ and $Tf$ are compact operators for every $f\in C_0(X)$, or equivalently, if $\chi_KT$ and $T\chi_K$ are compact operators for every compact $K\subseteq X$;
  \item $T$ is pseudolocal if $[f, T]$ is a compact operator for every $f\in C_0(X)$, or equivalently, if $\chi_K T\chi_F$ is a compact operator for every two disjoint compact subsets $K$ and $F$ of $X$;
  \item $T$ is locally compact on an open subset $A\subseteq X$ if $Tf$ and $fT$ are compact operators for every $f\in C_0(A)$, or equivalently, if $\chi_K T$ and $T\chi_K$ are compact operators for every compact $K\subseteq A$;
  \item $T$ is supported near a closed subset $Z\subseteq X$ if there is a constant $r>0$ such that $fT=Tf=0$ for every $f\in C_c(X)$ satisfying $d(\text{supp}(f), Z)>r$, or equivalently, if $T=\chi_{B(Z,r)}T\chi_{B(Z,r)}$ for some $r>0$ where $B(Z,r)$ is the closed neighborhood of $Z$ with distance less or equal to $r$;
  \item $T$ is called locally trace-class operator if $fT$ and $Tf$ are trace-class operators for all $f\in C_c( M)$, or equivalently, if $\chi_KT$ and $T\chi_K$ are trace-class operators for all compact subsets $K\subseteq M$.
\end{enumerate}
\end{defn}

\begin{defn}(cf. \cite{XieandYu14positivescalarcurvature})
For a geometric $X$-$G$ module $(\rho,H,U)$ and a $G$-invariant closed subset $Z\subseteq X$, the following $C^*$-algebras can be defined:
    \[
    C^*(X,H)^G=\overline{\{T\in B(H)^G:\text{locally compact, finite propagation} \}},
    \]
    \[
    D^*(X,H)^G=\overline{\{T\in B(H)^G:\text{pseudolocal, finite propagation}\}},
    \]
    \[
    C^*_Z(X,H)^G=\overline{\left\{T\in B(H)^G:\begin{array}{l}\text{locally compact, finite propagation, } \\ \text{ supported near $Z$}\end{array}\right\}},
    \]
    \[
    D^*_Z(X,H)^G=\overline{\left\{T\in B(H)^G:
    \begin{array}{l}
    \text{pseudolocal, finite propagation,}\\ \text{supported near $Z$, locally compact on $X-Z$} \end{array} \right\} }.
    \]
\end{defn}

Here overlines mean the norm-closure in $B(H)^G$. Note that $C^*_Z(X,H)^G$ is an ideal in $C^*(X,H)^G$, $D^*_Z(X,H)^G$ and $C^*(X,H)^G$ are ideals in $D^*(X,H)^G$. We may write above $C^*$-algebras as $C^*(X)^G$, $D^*(X)^G$, $C^*_Z(X)^G$ and $D^*_Z(X)^G$ when there is no ambiguity about the choice of the geometric $X$-$G$ modules. And we omit the symbol $G$ in the upper right corner if $G$ is a trivial group. The $C^*$-algebra $C^*(X,H)^G$ is often called the equivariant Roe algebra and $C^*_Z(X,H)^G$ the localized equivariant Roe algebra at $Z$ about the geometric $X$-$G$ module $(X,H,U)$ .


\begin{prop}(cf. \cite{Higson2000}, \cite{Rufus}) For geometric $X$-$G$ module $(\rho,H,U)$, geometric $X'$-$G$ module $(\rho',H',U')$, $G$-invariant closed subsets $Z\subseteq X$ and $Z'\subseteq X'$, if equivariant operator $T\in B(H,H')$ is an $\varepsilon$-cover of an equivariant coarse map $f:X\rightarrow X'$ satisfying $f(Z)\subseteq Z'$, then the *-homomorphism
\[
\text{ad}_T:B(H)^G\longrightarrow B(H')^G,\ S\longmapsto TST^*,
\]
maps $C^*(X,H)^G$ to $C^*(X',H')^G$ and $C_Z^*(X,H)^G$ to $C_{Z'}^*(X',H')^G$. The induced group homomorphisms in $K$-theory
\[
(\text{ad}_T)_*:K_*(C^*(X,H)^G)\longrightarrow K_*(C^*(X',H')^G),
\]
\[
(\text{ad}_T)_*:K_*(C^*_Z(X,H)^G)\longrightarrow K_*(C^*_{Z'}(X',H')^G),
\]
do not depend on the choice of the equivariant cover $T$ and can be denoted by $f_*$. If moreover $f$ is a uniform map and $T$ is a uniform cover, then $\text{ad}_T$ maps $D^*(X,H)^G$ to $D^*(X',H')^G$ and $D_Z^*(X,H)^G$ to $D_{Z'}^*(X',H')^G$. Similarly, the homomorphisms of groups
\[
(\text{ad}_T)_*:K_*(D^*(X,H)^G)\longrightarrow K_*(D^*(X',H')^G)
\]
\[
(\text{ad}_T)_*:K_*(D^*_Z(X,H)^G)\longrightarrow K_*(D^*_{Z'}(X',H')^G)
\]
do not depend on the choice of the equivariant uniform cover $T$ and can be denoted by $f_*$ as well.
\end{prop}

\subsection{Dirac Operators and Localized Equivariant Coarse Indices}\label{Dirac Operators and Localized Equivariant Coarse Indices}

\text{ }

In this subsection, we sketch the definitions of the Dirac bundls, the general Dirac operators, and the localized equivariant coarse indices. The way to get equivariant uniform covers by cutoff functions in this subsection is a fundamental trick to overcome the difficulties of defining the relative equivariant coarse index with group actions.

\begin{defn}(cf. \cite{Gromov}) Let $S$ be a bundle of Dirac modules over a Riemannian manifold $M$, that is, $S$ is a smooth bundle on $M$ of which each fiber $S_x$ for $x\in M$ is a left $Cl(T_xM)$-module and the Clifford multiplication $Cl(M)\times S\longrightarrow S$ is a smooth map. Here $Cl(T_xM)$ is the Clifford algebra generated by $T_xM$ and $Cl(M)=\coprod Cl(T_xM)$. Call $S$ a Dirac bundle if it is $\mathbb{Z}_2$-graded and equipped with a Hermitian metric and a compatible connection such that
\begin{enumerate}
        \item The Clifford action of each vector $v\in T_xM$ on $S_x$ is odd and skew-adjoint, that is, it exchanges the decomposition of $S$ and satisfies\\ $\langle v\cdot s_1,s_2\rangle+\langle s_1,v\cdot s_2\rangle =0$, $\forall s_1,s_2\in S_x$;
        \item The connection on $S$ is compatible with the Levi-Civita connection on $M$, in the sense that $\nabla_X(Ys)=(\nabla_XY)s+Y\nabla_X s$ for all vector fields $X,Y$ and sections $s\in C^\infty(M,S)$;
        \item  The connection and metric respect the decomposition of $S$.
\end{enumerate}
\end{defn}

\begin{defn}(cf. \cite{Gromov})
Let $M$ be a Riemannian manifold with a Dirac bundle $S$. Denote by $\{e_j\}$ an arbitrary smooth orthonormal tangent vector field on $M$.
\begin{enumerate}
\item The general Dirac operator $D$ on $S$ is locally defined by
\[
D:C_c^\infty(M,S)\longrightarrow C_c^\infty(M,S)
\]
\[
D=\sum e_j\cdot\nabla_{e_j};
\]
\item The connection Laplacian $\nabla^*\nabla$ on $S$ is locally defined by
\[
\nabla^*\nabla: C_c^\infty(M,S)\longrightarrow C_c^\infty(M,S)
\]
\[
\nabla^*\nabla=-\sum \left(\nabla_{e_j}\nabla_{e_j}-\nabla_{\nabla _{e_j}e_j}\right);
\]
\item The curvatrue operator $\mathcal{R}$ on $S$ is locally defined by
\[
\mathcal{R}:C_c^\infty(M,S)\longrightarrow C_c^\infty(M,S)
\]
\[
\mathcal{R}=\frac{1}{2}\sum e_i\cdot e_j \cdot R_{e_i,e_j},
\]
where 
\[
R_{X,Y}=\nabla_X\nabla_Y-\nabla_Y\nabla_X-\nabla_{[X,Y]},
\]
for every smooth tangent vector fields $X$ and $Y$ on $M$.
\end{enumerate}
\end{defn}

It is well-known that there is a Bocher-Weitzenbock formula \cite{Gromov} connecting above three operators on $S$
\[
D^2=\nabla^*\nabla+\mathcal{R}.
\]

Assume that $G$ is a countable discrete group acting on a complete Riemannian manifold $M$ with Dirac bundle $S$ properly by isometries. Also, suppose $S$ has isometric $G$-action and the projection $S\rightarrow M$ is equivariant. In this case, $H=L^2(M,S)$ is a separable Hilbert space, and the map $$\rho:C_0(X)\rightarrow B(H), f\mapsto \rho(f)$$ defined by $\rho(f)(\varphi)=f\varphi$ for every $\varphi\in H$, is a nondegenerate *-homomorphism mapping nonzero elements to noncompact operators. There is a group homomorphism $$U:G\rightarrow U(H), g\rightarrow U_g$$ where $U_g(\varphi)=\varphi\circ g^{-1}$ for every $\varphi \in H.$ Then $(\rho, H, U)$ is a geometric $M$-$G$ module since
      \[U_g\rho(f)U_g^*(\varphi)=U_g\rho(f)(\varphi\circ g)=(f\circ g^{-1})\varphi=\rho(f\circ g^{-1})(\varphi),\]
      \[\forall f\in C_0(X), g\in G,\varphi\in H.\]
Thanks to Proposition \ref{locallyfree}, we know that $(\rho\otimes 1, H\otimes l^2(G), U\otimes \lambda)$ is a locally free geometric $M$-$G$ module.

\begin{defn}(cf. \cite{Guo2019})
For a countable discrete group $G$ and a complete Riemannian manifold $M$ equipped with properly isometric $G$-actions, a smooth function $\omega$ on $M$ is called a cutoff function if it is nonnegative, its support has compact intersections with all $G$-orbits, and that for all $x\in M$,
\[
\sum_{g\in G}\omega(gx)^2=1.
\]
\end{defn}

Fix a cutoff function $\omega\in C^\infty(M)$. The map
\[
\iota:L^2(M,S)\longrightarrow L^2(M,S)\otimes l^2(G),
\]
given by
\[
(\iota(\varphi))(x,g)=\omega(g^{-1}x)\varphi(x),
\]
for $\varphi\in L^2(M,S)$, $x\in M$ and $g\in G$, is a $G$-equivariant, isometric embedding \cite{Hochs2019} and it intertwines the actions by $C_0(M)$ on $L^2(M,S)$ and $L^2(M,S)\otimes l^2(G)$. Moreover, $\iota$ is a uniform $0$-cover for $\text{id}_M$. Denote by $\mathcal{H}$ the Hilbert space $L^2(M,S)\otimes l^2(G)$ and identify $H$ with $\iota(H)$, then $\mathcal{H}=H\oplus H^\bot$ and
\[
\rho(f)\otimes1=\left(\begin{array}{cc}\rho(f) & 0 \\ 0 & *\end{array}\right),
\]
\[
U_g\otimes\lambda_g=\left(\begin{array}{cc} U_g  & 0 \\ 0 & *\end{array}\right),
\]
with respect to this direct decomposition of $\mathcal{H}$ for every $f\in C_0(M)$ and $g\in G$. The *-homomorphism
\[
\oplus\ 0:\ B(H)\longrightarrow B(\mathcal{H}), T\longmapsto T\oplus 0
\]
preserves finite propagation, local compactness, pseudolocality, having support near a $G$-invariant closed subset $Z\subseteq M$ and local compactness on $M-Z$, thus it induces
\[
\oplus\ 0: C^*(M,H)^G\longrightarrow C^*(M,\mathcal{H})^G;\ \ \ \ \oplus\ 0: C^*_Z(M,H)^G\longrightarrow C^*_Z(M,\mathcal{H})^G;
\]
\[
\oplus\ 0: D^*(M,H)^G\longrightarrow D^*(M,\mathcal{H})^G;\ \ \ \ \oplus\ 0: D^*_Z(M,H)^G\longrightarrow D^*_Z(M,\mathcal{H})^G.
\]
We still denote by $\oplus \ 0$ the group homomorphisms in $K$-theory derived from above *-homomorphisms. The maps $\oplus \ 0$ in $K$-theory level are independent of the choice of the cutoff function $\omega$. More precisely, the former two do not rely on the choice of equivariant cover of $\text{id}_M$ and the later two do not depend on the choice of equivariant uniform cover of $\text{id}_M$. Remark that
\[
\oplus\ 0: K_*(C^*(M,H)^G)\longrightarrow K_*(C^*(M,\mathcal{H})^G);
\]
\[
\oplus\ 0: K_*(C^*_Z(M,H)^G)\longrightarrow K_*(C^*_Z(M,\mathcal{H})^G);
\]
are group isomorphisms if $H$ itself is locally free.

Now, in order to introduce the localized equivariant coarse index of $D$ (cf. \cite{Guo2019} and \cite{Hochs2019}), assume that the curvature operator $\mathcal{R}$ is uniformly positive outside the $G$-invariant closed subset $Z$ of $M$, that is, there is a constant $\varepsilon>0$ such that
\[
\mathcal{R}_x\geq \varepsilon^2I,\ \ \forall x\in M-Z.
\]

  \begin{defn}
  An odd continuous function $\chi:\mathbb{R}\longrightarrow [-1,1]$ is called a normalizing function if $\lim_{t\rightarrow +\infty}\chi(t)=1$ and $\chi(t)>0$ for $t>0$.
  \end{defn}

  Take a normalizing function $\chi$ such that $\chi^2-1$ is supported in $(-\varepsilon,\varepsilon)$.

 When $\dim M=m$ is odd, lemma 2.3 in \cite{Roe2016} tells us that the equivalence of $\chi(D)$ in $D^*(M,H)^G/C^*_Z(M,H)^G$
 is independent of the choice of this kind of normalizing function $\chi$. Furthermore, the equivalence of $\frac{1+\chi(D)}{2}$ is a projection in $D^*(M,H)^G/C^*_Z(M,H)^G$.

 When $\dim M=m$ is even, the bundle $S$ has a $\mathbb{Z}_2$-grading  $S=S'\oplus S''$. In this case $D=\left(\begin{array}{ll} 0 & D_+ \\ D_- & 0\end{array}\right)$ and $
  \chi(D)=\left(\begin{array}{cc} 0 & \chi(D)_+ \\
  \chi(D)_- & 0
  \end{array}\right)
  $ with respect to this decomposition of $S$.
  Moreover, $\chi(D)_-$ is in $D^*(M,H)^G$ and the equivalence of $\chi(D)_-$ is a unitary in $D^*(M,H)^G/C^*_Z(M,H)^G$ which is independent of the choice of this kind of normalizing function $\chi$.

  Define
  \[
  [\chi(D)]=\left\{
              \begin{array}{ll}
                \left[ \frac{1+\chi(D)}{2} \right] \in K_{m+1}(D^*(M,H)^G/C^*_Z(M,H)^G), & \hbox{m is odd;} \\
                \left[\chi(D)_- \right]  \in K_{m+1}(D^*(M,H)^G/C^*_Z(M,H)^G), & \hbox{m is even.}
              \end{array}
            \right.
  \]
  The exact sequences of $C^*$-algebras
  \[
  \xymatrix{
    0  \ar[r] & C^*_Z(M,H)^G \ar[d]^{\oplus\ 0} \ar[r] & D^*(M,H)^G \ar[d]^{\oplus\ 0} \ar[r] & D^*(M,H)^G/C^*_Z(M,H)^G \ar[d]^{\oplus\ 0} \ar[r] & 0  \\
    0 \ar[r] & C^*_Z(M,\mathcal{H})^G \ar[r] & D^*(M,\mathcal{H})^G \ar[r] & D^*(M,\mathcal{H})^G/C^*_Z(M,\mathcal{H})^G \ar[r] & 0   }
  \]
  induces a commutative diagram in $K$-theory:
  \[
  \xymatrix{
    K_{m+1}\left(D^*(M,H)^G/C^*_Z(M,H)^G\right) \ar[d]^{\oplus\ 0} \ar[r]^{\ \ \ \ \ \ \ \ \ \ \ \ \  \partial} & K_{m}(C^*_Z(M,H)^G) \ar[d]^{\oplus\ 0} \\
    K_{m+1}\left(D^*(M,\mathcal{H})^G/C^*_Z(M,\mathcal{H})^G\right) \ar[r]^{\ \ \ \ \ \ \ \ \ \ \ \ \  \partial'} & K_{m}(C^*_Z(M,\mathcal{H})^G).   }
  \]
  The localized equivariant coarse index of $D$ is defined by
  \[
  \text{Index}^G_Z (D)=\left(\partial [\chi(D)]\right)\oplus 0=\partial'\left( [\chi(D)]\oplus 0\right)\ \in K_m(C^*_Z(M,\mathcal{H})^G).
  \]
\subsection{Computation in $K$-Theory}

\text{ }

In this subsection, we introduce some computation method in $K$-theory that can be used to compute our relative equivariant coarse index defined in Section \ref{The Definition of the Relative Equivariant Coarse Index} and to compute the relative $L^2$-index defined in Section \ref{Relative L2 Index}. This result is a generalization of the exercise 2.11.22 in \cite{Rufus}.

Let $C$ be a $C^{\ast}$-algebra and $I$ a closed ideal in $C$. The double of $C$ along $I$ is the $C^{\ast}$-algebra defined by
	$$D_C(I)=\{(a,b)\in C\oplus C\mid a-b\in I\}.$$

The natural inclusion $$I\longrightarrow D_{C}(I),\ \ a\longmapsto (a,0),$$
and the natural projection
\[
D_C(I)\longrightarrow C,\ \ (a,b)\longmapsto b
\]
lead to a split exact sequence of $C^*$-algebras
\[
\xymatrix@C=0.5cm{
  0 \ar[r] & I \ar[rr] && D_C(I) \ar[rr] && C \ar[r] & 0 }
\]
with the splitting map
\[
C\longrightarrow  D_C(I),\ \ a \longrightarrow (a,a),
\]
and thus a direct sum decomposition
	$$K_*(D_{C}(I))\cong K_*(I)\oplus K_*(C).$$
	
    Regard $M_{n\times n}(D_C(I)^+)$ as a subset of $M_{n\times n}(C^+)\oplus M_{n\times n}(C^+)$ for every positive integer $n$ by the isomorphism
    \[
    M_{n\times n}(D_C(I)^+)\cong \left\{(x,y)\in M_{n\times n}(C^+)\oplus M_{n\times n}(C^+):x-y\in M_{n\times n}(I) \right\}.
    \]
    For every idempotents $(e,f)$ in $M_{n\times n}(D_C(I)^+)$, obviously $e$ and $f$ are idempotents in $M_{n\times n}(C^+)$. Let
    $$\mathcal{Z}(f)=\begin{pmatrix} f & 0&  1-f & 0\\ 1-f &0 & 0 &f\\ 0 & 0 & f &1-f \\ 0 &1 &0 &0\end{pmatrix},$$
    whose inverse is
    $$\mathcal{Z}(f)^{-1}=\begin{pmatrix} f & 1-f &0& 0\\ 0 &0 & 0 &1\\ 1-f & 0 & f &0 \\ 0 &f &1-f &0\end{pmatrix}.$$
    Then define
        \begin{align*}
        \mathcal{E}(e,f)&=\mathcal{Z}(f)^{-1}\begin{pmatrix}e & 0 & 0 &0\\ 0 & 1-f &0 & 0\\ 0 & 0 & 0 &0\\ 0 & 0 & 0 &0\end{pmatrix}\mathcal{Z}(f)\\
        & =\begin{pmatrix}
        fef+(1-f) & 0 & fe(1-f) & 0\\
        0 & 0& 0& 0\\
        (1-f)ef & 0 & (1-f)e(1-f) & 0 \\
        0 & 0& 0& 0
        \end{pmatrix}
        \\
       & =\begin{pmatrix} 1+ f(e-f)f & 0 & fe(e-f) & 0\\
        0 & 0 &0 &0\\ (e-f)ef &0 & (1-f)(e-f)(1-f) &0 \\ 0 &0 &0 &0\end{pmatrix},
        \end{align*}
    and
    \[\mathcal{E}_n=
    \begin{pmatrix}1 &0 & 0 &0\\ 0 &0 &0 &0\\ 0&0 &0 &0\\ 0 &0 &0 &0\end{pmatrix}.\]
    It follows immediately that
    \[\mathcal{E}(e,f)\in M_{4n\times 4n}(I^+),\ \ \ \ \mathcal{E}_n\in M_{4n\times 4n}(I^+),\ \ \ \ \mathcal{E}(e,f)-\mathcal{E}_n\in M_{4n\times 4n}(I).\]
    Furtherly, $\mathcal{E}(e,f)$ and $\mathcal{E}_n$ are idempotents.


    \begin{prop}\label{Lemma:DifferenceConstruction}
    For every element $x$ in $K_0(D_C(I))$, there are idempotents $(e,f)$ and $(e',f')$ in $M_{n\times n}(D_C(I)^+)$ such that
     \[
     x=[(e,f)]-[(e',f')]\in K_0(D_C(I)),
     \]
     \[
     (e,f)-(e',f')\in M_{n\times n}(D_C(I)).
     \]
     According to the decomposition $K_0(D_{C}(I))\cong K_0(I)\oplus K_0(C)$, decompose $x$ into $K_0(I)$ is exactly
        \[
        [\mathcal{E}(e,f)]-[\mathcal{E}(e',f')]\in K_0(I).
        \]
    \end{prop}
    \begin{proof}
    Consider the split exact sequence in $K$-theory
    \[
    \xymatrix{
            0 \ar[r]& K_0(I) \ar[r] &K_0(D_C(I))\ar[r] & K_0(C)\ar[r]& 0
        }.
    \]
    Denote by $\alpha$ and $\beta$ the inclusion map and the splitting map respectively.

    Note that $(1-f,1-f)$ and $(1-f',1-f')$ are idempotents in $M_{n\times n}(D_C(I)^+)$, $(\mathcal{Z}(f),\mathcal{Z}(f))$ and $(\mathcal{Z}(f'),\mathcal{Z}(f'))$ are invertibles in $M_{n\times n}(D_C(I)^+)$, $1-f$ and $1-f'$ are idempotents in $M_{n\times n}(C^+)$. There are \[
    [(1-f,1-f)]-[(1-f',1-f')]\in K_0(D_C(I)),\]
    \[
        [\mathcal{E}(e,f)]-[\mathcal{E}(e',f')]\in K_0(I),\ [1-f]-[1-f']\in K_0(C),
        \]
    and the following equations hold in $K_0(D_C(I))$:
    \begin{align*}
            &[(e,f)]-[(e',f')]&\\
            =&\left[\left(e\oplus (1-f),f\oplus (1-f)\right)\right]-[(1-f,1-f)]&\\
            -&\left[\left(e'\oplus (1-f'),f'\oplus (1-f')\right)\right]+[(1-f',1-f')]&\\
            =&\left[(\mathcal{Z}(f),\mathcal{Z}(f))^{-1}\left(D_{e,1-f}, D_{f,1-f}\right)(\mathcal{Z}(f),\mathcal{Z}(f))\right]-[(1-f,1-f)]&\\
            -&\left[(\mathcal{Z}(f'),\mathcal{Z}(f'))^{-1}\left(D_{e',1-f'}, D_{f',1-f'}\right)(\mathcal{Z}(f'),\mathcal{Z}(f'))\right]+[(1-f',1-f')]&\\
            =&[(\mathcal{E}(e,f),\mathcal{E}_n)]-[(1-f,1-f)]-[(\mathcal{E}(e',f'),\mathcal{E}_n)]+[(1-f',1-f')]&\\
            =&\alpha([\mathcal{E}(e,f)]-[\mathcal{E}(e',f')])+\beta([1-f']-[1-f]),&
        \end{align*}
        where
        $$
        D_{e,1-f}=\begin{pmatrix}e & 0 & 0 &0\\ 0 & 1-f &0 & 0\\ 0 & 0 & 0 &0\\ 0 & 0 & 0 &0\end{pmatrix}, \quad D_{f,1-f}=\begin{pmatrix} f&0 &0 &0\\ 0&1-f&0 &0\\ 0&0 & 0&0\\ 0&0 &0 &0\end{pmatrix},$$
        and similarly for $D_{e',1-f'}$ and $D_{f',1-f'}$. The proposition has been proved.
    \end{proof}

\section{Relative Equivariant Coarse Index}\label{The Definition of the Relative Equivariant Coarse Index}

In this section, we define the relative equivariant coarse index with proper actions and give some further expositions about alternative ways to get it and the property that the resulting index will not change by different choice in the process.

\subsection{Definition of Relative Equivariant Coarse Index}

\text{ }

The main purpose of this subsection is to present the steps to define the relative equivariant coarse index.

The conditions that we need is summarised as a set of relative equivariant coarse index data:
\begin{enumerate}
  \item[(i)] A countable discrete group $G$ acts on complete Riemannian manifolds $M_1$ and $M_2$ properly by isometries. The manifolds $M_i$ are equipped with the Dirac operators $D_i$ on Dirac bundles $S_i$ and $Y_i\subseteq M_i$ are $G$-invariant closed subsets. The group $G$ also acts on $S_i$ by isometries and the projections $S_i\rightarrow M_i$ are $G$-equivariant. 
  \item[(ii)]A $G$-equivariant isometric diffeomorphism $h:M_1-Y_1\rightarrow M_2-Y_2$ covered by bundle isomorphism $\bar{h}$ satisfies
\[
\bar{h}\circ(D_1\varphi)\circ h^{-1}=D_2(\bar{h}\circ \varphi \circ h^{-1}),\forall \varphi\in L^2(M_1-Y_1,S_1).
\]
\item[(iii)]A $G$-equivariant coarse map $q:Y_2\longrightarrow Y_1$ satisfies that the map $M_2\longrightarrow M_1$ defined by $h^{-1}$ on $M_2-Y_2$ and $q$ on $Y_2$ is still a $G$-equivariant coarse map.
\end{enumerate}

These can be simply explained as: for two complete Riemannian manifolds $M_i$ equipped with Dirac operators $D_i$ and proper isometric $G$-actions, there exists an equivariant coarse map from $M_2$ to $M_1$ satisfying that it preserves all the structures between $G$-invariant open subsets $M_i-Y_i$ and it maps $Y_2$ to $Y_1$.

Based on a set of relative equivariant coarse index data $(M_i,S_i,D_i,Y_i,h,q)$ over $G$ defined above, let
$$H_i=L^2(M_i,S_i),$$
 $$\rho_i:C_0(M_i)\rightarrow B(H_i), f\mapsto \rho_i(f),$$
 $$U_{i}:G\rightarrow U(H_i),g\mapsto U_{ig},$$ where $
 \rho_i(f)(\varphi)=f\varphi$ and $U_{ig}\varphi=\varphi\circ g^{-1}$ for every $\varphi \in H_i$ and $i=1,2$.
According to the talk in Section \ref{Dirac Operators and Localized Equivariant Coarse Indices}, we know that $(\rho_i,H_i,U_i)$ are geometric $M_i$-$G$ modules. Proposition \ref{locallyfree} tells us that $(\rho_i\otimes 1,H_i\otimes l^2(G),U_i\otimes \lambda)$ are locally free geometric $M_i$-$G$ modules. Let $\mathcal{H}_i$$=H_i\otimes l^2(G)$.

\begin{lem}\label{nbhd}
For every $t>0$, the map $h$ is bijective between $O(Y_1,t)-Y_1$ and $O(Y_2,t)-Y_2$. Here $O(Y_i,t)$ mean the open neighborhood of $Y_i$ with distance less than $t$ which are obviously $G$-invariant.
\end{lem}
\begin{proof}
Take $x\in O(Y_1,t)-Y_1$, then $d(x, Y_1)<t$ and there is $y\in Y_1$ such that $d(x,y)=d(x,Y_1)<t$. We can find a Cauchy sequence $\{y_n\}$ in $O(Y_1,t)-Y_1$ converging to $y$ according to the existence of minimal geodesic between $x$ and $y$. So $\{h(y_n)\}$ converge to a point $y'$ and $y'$ must be in $Y_2$. Otherwise, there is $y\notin Y_1$ which is a contradiction. Consequently, $d(h(x), Y_2)\leq d(h(x),y')=\lim d(h(x),h(y_n))=\lim d(x,y_n)=d(x,y)<t$.
\end{proof}

 Fix a constant number $s>0$ and let $Z_i=\overline{O(Y_i,s)}$. Lemma \ref{nbhd} guarantees that $h:M_1-Z_1\longrightarrow M_2-Z_2$ is a $G$-equivariant isometric diffeomorphism.

\begin{lem}\label{cutofffunction}
There are cutoff functions $\omega_1$ on $M_1$ and $\omega_2$ on $M_2$ such that the diagram commutative
\[
\xymatrix{
  M_1-Z_1 \ar[rr]^{h} \ar[dr]_{\omega_1}
                &  &    M_2-Z_2 \ar[dl]^{\omega_2}    \\
                & \mathbb{R}                 }
\]
\end{lem}

\begin{proof}
As state in \cite{Guo2019}, there are cutoff functions $\theta_i$ on $O(Y_i,s)$, $\xi_i$ on $M_i-Y_i$ for $i=1,2$. We can let $\xi_1=\xi_2\circ h$ on $M_1-Y_1$. Take a partition of unit $\{\alpha_1,\beta_1\}$ on $M_1$ subordinate to $\{O(Y_1,s), M_1-Y_1\}$ and a partition of unit $\{\alpha_2,\beta_2\}$ on $M_2$ subordinate to $\{O(Y_2,s), M_2-Y_2\}$ which are $G$-invariant, that is,
\[
\alpha_i(gx)=\alpha_i(x),\forall x\in O(Y_i,s),g\in G,i=1,2;
\]
\[
\beta_i(gx)=\beta_i(x),\forall x\in M_i-Y_i,g\in G,i=1,2.
\]
Define
\[
\omega_1=\sqrt{\alpha_1 \theta_1^2+\beta_1\xi_1^2},
\]
\[
\omega_2=\sqrt{\alpha_2 \theta_2^2+\beta_2\xi_2^2}.
\]
Then $\omega_i\in C^\infty(M_i)$ and their supports have compact intersections with all $G$-orbits. For all $x\in M_i$,
\[
\sum_G \omega_i(gx)^2 =\sum_G \left(\alpha_i (gx)\theta_i(gx)^2+\beta_i(gx)\xi_i(gx)^2 \right)=\alpha_i(x)+\beta_i(x)=1.
\]
On $M_i-Z_i$, there are $\omega_i=\xi_i$ and then the diagram above is commutative.
\end{proof}

Using the cutoff functions $\omega_i$ in lemma \ref{cutofffunction}, there are
\[
\iota_i : \ H_i\longrightarrow \mathcal{H}_i,\ \iota_i(\varphi)(x,g)=\omega_i(g^{-1}x)\varphi(x),
\]
for all $\varphi\in H_i$, $x\in M_i$, $g\in G$. Equivalently,
\[
\iota_i(\varphi)=\sum_{g\in G}\varphi_g\otimes \delta_g,\forall \varphi\in H_i,
\]
where
\[
\varphi_g(x)=\omega_i(g^{-1}x)\varphi(x), \forall x\in M_i.
\]
Recall that $\iota_i$ are $G$-equivariant, isometric embedding \cite{Hochs2019} and they intertwines the actions by $C_0(M)$ on $H_i$ and $\mathcal{H}_i$. Identify $H_i$ with $\iota_i(H_i)$, then $\mathcal{H}_i=H_i\oplus H_i^\bot$ and there are
\[
\rho_i(f)\otimes1=\left(\begin{array}{cc}\rho_i(f) & 0 \\ 0 & *\end{array}\right),
\]
\[
U_{ig}\otimes\lambda_g=\left(\begin{array}{cc} U_{ig}  & 0 \\ 0 & *\end{array}\right),
\]
with respect to this decomposition of $\mathcal{H}_i$ for $f\in C_0(M_i)$, $g\in G$.

Define
  \[
  V:L^2(M_1-Z_1,S_1)\longrightarrow L^2(M_2-Z_2,S_2)
  \]
  \[
  \varphi\longmapsto \bar{h}\circ \varphi\circ h^{-1}
  \]
  which is a $G$-equivariant unitary operator. The zero extension of $V$ from $H_1$ to $H_2$ is defined by
  \[
  V:L^2(M_1,S_1)\longrightarrow L^2(M_2,S_2)
  \]
  \[
  \varphi\longmapsto \bar{h}\circ (\varphi\cdot \chi_{Z_1^c})\circ h^{-1}
  \]
  where $Z_1^c=M_1-Z_1$ and we regard $\bar{h}\circ (\varphi\cdot \chi_{Z_1^c})\circ h^{-1}$ equals to 0 on $Y_2$. Then there is a $G$-equivariant partial isometry $V\otimes 1$ from $H_1\otimes l^2(G)$ to $H_2\otimes l^2(G)$, that is, from $\mathcal{H}_1$ to $\mathcal{H}_2$. 
  \begin{lem}
  The operator $V:H_1\longrightarrow H_2$ has the following property:
  \[
  \rho_2(f)V=V\rho_1(f\circ h),   \forall f\in B(M_2).
  \]Note that we regard $f\circ h$ as 0 on $Y_1$.
  \end{lem}
  \begin{proof}
  Take an element $\varphi \in L^2(M_1,S_1)$.
  \[
  \rho_2(f)V(\varphi)=\rho_2(f)\left(\bar{h}\circ (\varphi\cdot \chi_{Z_1^c})\circ h^{-1}\right)=f\cdot\left(\bar{h}\circ (\varphi\cdot \chi_{Z_1^c})\circ h^{-1}\right),
  \]
  \[
  V\rho_1(f\circ h)(\varphi)=V\left((f\circ h)\cdot \varphi\right)=\bar{h}\circ \left((f\circ h)\cdot \varphi\cdot \chi_{Z_1^c}\right)\circ h^{-1}.
  \]
  Then
  \[
  \rho_2(f)V=V\rho_1(f\circ h).
  \]
  \end{proof}
  \begin{lem}\label{cmap}
  The following map
  \[
  \Phi :C^*(M_1,\mathcal{H}_1)^G/C^*_{Z_1}( M_1,\mathcal{H}_1)^G\longrightarrow C^*(M_2,\mathcal{H}_2)^G/C^*_{Z_2}( M_2,\mathcal{H}_2)^G,
  \]
  \[
  [T]\longmapsto [(V\otimes 1) T (V^*\otimes 1)],
  \]
  is a well-defined *-isomorphism.
  \end{lem}

  \begin{proof}



        If $T\in B(\mathcal{H}_1)^G$, then $(V\otimes 1) T (V^*\otimes 1)\in B(\mathcal{H}_2)^G$ obviously.

         Suppose $T\in B(\mathcal{H}_1)^G$ is locally compact, then $(V\otimes 1) T (V^*\otimes 1)\in B(\mathcal{H}_2)^G$ is also locally compact since for every $f\in C_c(M_2)$, there is
        \[
          \begin{array}{l}
            (\rho_2(f)\otimes 1)(V\otimes 1) T (V^*\otimes 1)
           =  (V\otimes 1)(\rho_1((f\circ h)\cdot\chi_{\overline{Z_1^c}})\otimes 1) T (V^*\otimes 1),\\
            (V\otimes 1) T (V^*\otimes 1)(\rho_2(f)\otimes 1)
           =  (V\otimes 1) T (\rho_1(f\circ  h)\otimes 1)(V^*\otimes 1),
          \end{array}
        \]
        and $(f\circ  h)\cdot \chi_{\overline{Z_1^c}}$ is compactly supported.

        Given $T\in B(\mathcal{H}_1)^G$ and prop$(T)<r<\infty$, then $(V\otimes 1) T (V^*\otimes 1)\in B(\mathcal{H}_2)^G$ has finite propagation.
        In fact, it is easy to find that
        \[
        \text{supp}(V\otimes 1) T (V^*\otimes 1)\subseteq \overline{M_2-Z_2}\times\overline{M_2-Z_2}.
        \]
        Take $(y,y')\in \overline{M_2-Z_2}\times\overline{M_2-Z_2}$ and $d(y,y')>r$. There are $x,x'\in \overline{M_1-Z_1}$ such that $h(x)=y$ and $h(x')=y'$. Clearly, $d(x,x')>r$ which means there are open neighborhoods $A$ and $A'$ for $x$ and $x'$ respectively such that $\chi_A T\chi_{A'}=0$. We can suppose that $A$ and $A'$ have compact closure and $d(A,A')>r$. Let $$B=h(A\cap\overline{M_1-Z_1}),$$
        $$B'=h(A'\cap\overline{M_1-Z_1}).$$
        Then $$\chi_B(V\otimes 1)T(V^*\otimes 1)\chi_{B'}=(V\otimes 1)\chi_AT\chi_{A'}(V^*\otimes 1)=0.$$
        Since $B$ and $B'$ are relative open sets in $\overline{M_2-Z_2}$, there are open neighborhoods $C$ and $C'$ for $y$ and $y'$ in $M_2$ such that
        \[
        C\cap\overline{M_2-Z_2}=B,
        \]
        \[
        C'\cap\overline{M_2-Z_2}=B'.
        \]
        Consequently, $$\chi_C(V\otimes 1)T(V^*\otimes 1)\chi_{C'}=\chi_B(V\otimes 1)T(V^*\otimes 1)\chi_{B'}=0.$$
        This shows that
        \[
        (y,y')\notin \text{supp}(V\otimes 1) T (V^*\otimes 1),
        \]
        and
        \[
        \text{prop}(V\otimes 1) T (V^*\otimes 1)\leq r.
        \]

         Assume $T\in B(\mathcal{H}_1)^G$ is supported near $Z_1$, then $(V\otimes 1) T (V^*\otimes 1)\in B(\mathcal{H}_2)^G$ is supported near $Z_2$. In fact, there is $r>0$ such that
        \[
        (\rho_1(f)\otimes 1)T=T(\rho_1(f)\otimes 1)=0, \forall f\in C_c(M_1), d(\text{supp}(f), Z_1)>r.
        \]
        For every $f\in C_c(M_2)$ satisfying $d(\text{supp}(f), Z_2)>r$, we have
        \[
        (\rho_2(f)\otimes1)(V\otimes 1) T (V^*\otimes 1)=(V\otimes 1)(\rho_1(f\circ h)\otimes 1 )T (V^*\otimes 1)=0,
        \]
        \[
        (V\otimes 1) T (V^*\otimes 1)(\rho_2(f)\otimes 1)=(V\otimes 1) T (\rho_1(f\circ h)\otimes 1)(V^*\otimes 1)=0.
        \]

        By now, we have shown that the bounded linear *-preserving function
        $$B(\mathcal{H}_1)^G\longrightarrow B(\mathcal{H}_2)^G$$
        $$T\longmapsto (V\otimes 1) T (V^*\otimes 1)$$ maps $C^*(M_1,\mathcal{H}_1)^G$ to $C^*(M_2,\mathcal{H}_2)^G$ and $C^*_{Z_1}(M_1,\mathcal{H}_1)^G$ to $C^*_{Z_2}(M_2,\mathcal{H}_2)^G$. So $\Phi$ is well-defined and is a bounded linear *-preserving map.

        Take locally compact operators $T,S\in B(H_1)^G$ with propagations less than $r<\infty$. Then
        \[
         (V\otimes 1) TS (V^*\otimes 1)- (V\otimes 1) T (V^*\otimes 1) (V\otimes 1) S (V^*\otimes 1)
        \]
        \[
         =(V\otimes 1)\left( TS-T\chi_{Z_1^c}
          S \right) (V^*\otimes 1)
        \]
        \[
         =(V\otimes 1) T\chi_{Z_1}S (V^*\otimes 1)
        \]
        is $G$-equivariant, locally compact and has finite propagation by previous analysis. For every $f\in C_c(M_2)$, $d(\text{supp}(f), Z_2)>r$, there are
        \[
        (\rho_2(f)\otimes 1)(V\otimes 1) T\chi_{Z_1}S (V^*\otimes 1)=(V\otimes 1) (\rho_1(f\circ  h)\otimes 1)T\chi_{Z_1}S (V^*\otimes 1)=0,
        \]
        \[
        (V\otimes 1) T\chi_{Z_1}S (V^*\otimes 1)(\rho_2(f)\otimes 1)=(V\otimes 1) T\chi_{Z_1}S(\rho_1(f\circ  h)\otimes 1) (V^*\otimes 1)=0,
        \]
        which mean that $(V\otimes 1) T\chi_{Z_1}S(V^*\otimes 1)$ is supported near $Z_2$. Consequently,
        \[ \Phi([T]\cdot[S])=\Phi([T])\cdot\Phi([S]),
        \]
        that is, the map $\Phi$ preserves the product operations on $C^*(M_1,\mathcal{H}_1)^G/C^*_{Z_1}( M_1,\mathcal{H}_1)^G$ and $C^*(M_2,\mathcal{H}_2)^G/C^*_{Z_2}( M_2,\mathcal{H}_2)^G$.

        As a result, $\Phi$ is a *-homomorphism and it is further a *-isomorphism according to the symmetry between $M_1-Z_1$ and $M_2-Z_2$.
  \end{proof}

  \begin{lem}\label{dmap}
  The following map
  \[
   \Psi:D^*(M_1,\mathcal{H}_1)^G/D^*_{Z_1}( M_1,\mathcal{H}_1)^G\longrightarrow D^*(M_2,\mathcal{H}_2)^G/D^*_{Z_2}( M_2 ,\mathcal{H}_2)^G,
  \]
  \[
  [T]\longmapsto [(V\otimes 1) T (V^*\otimes 1)],
  \]
  is a well-defined *-isomorphism.
  \end{lem}

  \begin{proof}
   If $T\in B(\mathcal{H}_1)^G$ is pseudolocal, then so is $(V\otimes 1)T(V^*\otimes 1)\in B(\mathcal{H}_2)^G$. In fact, for every two disjoint compact subset $K$ and $F$ of $M_2$, let
   \[
   K'=h^{-1}(K\cap\overline {M_2-Z_2}),
   \]
   \[
   F'=h^{-1}(F\cap\overline {M_2-Z_2}),
   \]
   which are disjoint compact subsets of $M_1$. There is
   \[
   \chi_K(V\otimes 1)T(V^*\otimes 1)\chi_F=(V\otimes 1)\chi_{K'} T\chi_{F'}(V^*\otimes 1)
   \]
   which is a compact operator due to $T$ is pseudolocal.

   Assume $T\in B(\mathcal{H}_1)^G$ is locally compact on $M_1-Z_1$, then $(V\otimes 1)T(V^*\otimes 1)$ is locally compact on $M_2-Z_2$. In fact, for every $f\in C_c(M_2-Z_2)$, the operators
   \[
   (\rho_2(f)\otimes 1)(V\otimes 1)T(V^*\otimes 1)=(V\otimes 1)(\rho_1(f\circ  h)\otimes 1)T(V^*\otimes 1),
   \]
   \[
   (V\otimes 1)T(V^*\otimes 1)(\rho_2(f)\otimes 1)=(V\otimes 1)T(\rho_1(f\circ  h)\otimes 1)(V^*\otimes 1),
   \]
   are compact since $T$ is locally compact on $M_1-Z_1$.

   Combined with the proof of the former lemma, we have known that the bounded linear *-preserving function
        $$B(\mathcal{H}_1)^G\longrightarrow B(\mathcal{H}_2)^G$$
        $$T\longmapsto (V\otimes 1) T (V^*\otimes 1)$$ maps $D^*(M_1,\mathcal{H}_1)^G$ to $D^*(M_2,\mathcal{H}_2)^G$ and $D^*_{Z_1}(M_1,\mathcal{H}_1)^G$ to $D^*_{Z_2}(M_2,\mathcal{H}_2)^G$. So $\Psi$ is well-defined and is a bounded linear *-preserving map.

        Take pseudolocal operators $T,S\in B(H_1)^G$ with propagations. Then
        \[
         (V\otimes 1) TS (V^*\otimes 1)- (V\otimes 1) T (V^*\otimes 1) (V\otimes 1) S (V^*\otimes 1)
        \]
        \[
         =(V\otimes 1)\left( TS-T\chi_{Z_1^c}S\right) (V^*\otimes 1)
        \]
        \[
         =(V\otimes 1) T\chi_{Z_1}S (V^*\otimes 1)
        \]
        is equivariant, pseudolocal, locally compact on $M_2-Z_2$, supported near $Z_2$ and has finite propagation by previous analysis. Consequently,
        \[
         \Psi([T]\cdot[S])=\Psi([T])\cdot\Psi([S]),
        \]
        that is, the map $\Psi$ preserves the product operations on $D^*(M_1,\mathcal{H}_1)^G/D^*_{Z_1}( M_1,\mathcal{H}_1)^G$ and $D^*(M_2,\mathcal{H}_2)^G/D^*_{Z_2}( M_2,\mathcal{H}_2)^G$.

        As a result, the map $\Psi$ is a *-homomorphism and it is further a *-isomorphism according to the symmetry between $M_1-Z_1$ and $M_2-Z_2$.
  \end{proof}

  \begin{lem}\label{V}
  There is the following commutative graph:
  \[
  \xymatrix{
    H_1 \ar[d]_{\iota_1} \ar[r]^{V} &H_2 \ar[d]^{\iota_2} \\
    \mathcal{H}_1\ar[r]^{V\otimes 1} & \mathcal{H}_2.   }
  \]
  Equivalently,
  \[
  V\otimes 1=\left(\begin{array}{cc} V & 0 \\ 0 & *\end{array}\right)
  \]
  according to the decompositions $\mathcal{H}_i=H_i\oplus H_i^\bot$ for $i=1$, $2$.
  \end{lem}

  \begin{proof}
  Take an element $\varphi\in H_1$, then
  \[
  \varphi\overset{\iota_1}{\longmapsto}\sum_{g\in G}\varphi_g\otimes \delta_g \overset{V\otimes 1}{\longmapsto} \sum_{g\in G}V\varphi_g\otimes \delta_g;
  \]
  \[
  \varphi\overset{V}{\longmapsto} \psi \overset{\iota_2}{\longmapsto} \sum_{g\in G} \psi_g \otimes \delta_g.
  \]
  Here
  \[
  \begin{array}{cc}
  \varphi_g(x)=\omega_1(g^{-1}x)\varphi(x), \forall x\in M_1; & V\varphi_g=\bar{h}\circ (\varphi_g \cdot \chi_{Z_1^c})\circ h^{-1}; \\
  \psi=\bar h\circ (\varphi \cdot \chi_{Z_1^c})\circ h^{-1}; &   \psi_g(y)=\omega_2(g^{-1}y)\psi(y),\forall y\in M_2.\\
  \end{array}
  \]
  Just need to check that $V\varphi_g=\psi_g$ for all $g\in G$. Obviously, we can know that $V\varphi_g=\psi_g=0$ on $Z_2$. For every $y\in M_2-Z_2$, there is a unique $x\in M_1-Z_1$ such that $h(x)=y$, then
  \[
  V\varphi_g(y)=\bar{h}\circ \varphi_g(x)=\bar h (\omega_1(g^{-1}x)\varphi(x)) =\omega_1(g^{-1}x)\bar h (\varphi(x)),
  \]
  \[
  \psi_g(y)=\omega_2(g^{-1}y)\psi(y)=\omega_2(g^{-1}y)\bar h(\varphi(x))=\omega_1(g^{-1}x)\bar h(\varphi(x)).
  \]
  As a result, $V\varphi_g=\psi_g$ on $M_2$ for every $g\in G$ and the above diagram is commutative.
  \end{proof}

  \begin{lem}\label{c0function}
  For every $f\in C_0(\mathbb{R})$, there is $\Phi[f(D_1)\oplus 0]=[f(D_2)\oplus 0]$.
  \end{lem}

  \begin{proof}
  It is known that $f(D_i)\in C^*(M_i,H_i)^G$ for $f\in C_0(\mathbb{R})$ and $i=1,2$ (cf. \cite{Higson2000}, \cite{Rufus}). So $f(D_i)\oplus 0$ are in $ C^*(M_i,\mathcal{H}_i)^G$ under the map
  \[
  \oplus \ 0: C^*(M_i,H_i)^G\longrightarrow C^*(M_i,\mathcal{H}_i)^G.
  \]
  Just need to prove that $$\Phi[f(D_1)\oplus 0]=[f(D_2)\oplus 0],$$ for every $f\in C_0(\mathbb{R})$ which has compactly supported Fourier transform since this kind of functions comprise a dense subset of $C_0(\mathbb{R})$.

  Suppose $f\in C_0(\mathbb{R})$ has Fourier transform supported in $(-r,r)$. Take smooth functions
  \[
  \psi_i:M_i\rightarrow [0,1]
  \]
  for $i=1,2$ satisfying
  \[
  \psi_i|O(Z_i,r)\equiv 1,\text{supp}(\psi_i)\subseteq O(Z_i,2r),
  \]
  and
  \[
  \psi_1=\psi_2\circ h \text{  on  } M_1-Z_1.
  \]
  We have the following decomposition
  \[
  f(D_i)=f(D_i)\rho_i(\psi_i)+f(D_i)(1-\rho_i(\psi_i)).
  \]
  The fact
  \[
  f(D_i)\rho_i(\psi_i)\in C^*_{Z_i}(M_i, H_i)^G, \ \ i=1,2,
  \]
  holds since
  \[
  \text{prop}\left(f(D_i)\right)<r, \ \ i=1,2.
  \]
  So
  \[
  (f(D_i)\rho_i(\psi_i))\oplus 0\in  C^*_{Z_i}(M_i, \mathcal{H}_i)^G.
  \]
  Then
  \[
  [f(D_i)\oplus 0]=[\left(f(D_i)\rho_i(1-\psi_i)\right) \oplus 0].
  \]
  According to the definition of the operator $V$ and the relationship between $D_1$ and $D_2$, $\psi_1$ and $\psi_2$, there is
  \[
    V \text{e}^{itD_1}\rho_1(1-\psi_1)V^*=  V\text{e}^{itD_1} V^*V\rho_1(1-\psi_1)V^*=  \text{e}^{itD_2}\rho_2(1-\psi_2),
  \]
  for $t\in (-r,r)$.
  This tell us that
  \[
  \begin{array}{ll}
  & \langle f(D_2)\rho_2(1-\psi_2)x\ ,\ y\rangle \\
  =&\frac{1}{2\pi}\displaystyle \int^r_{-r} \widehat{f}(t)\langle\text{e}^{itD_2}\rho_2(1-\psi_2)x,y\rangle dt \\
  =& \frac{1}{2\pi}\displaystyle\int^r_{-r} \widehat{f}(t)\langle\text{e}^{itD_1}\rho_1(1-\psi_1)V^*x, V^*y\rangle dt \\
  =& \langle f(D_1)\rho_1(1-\psi_1)V^*x\ ,\ V^*y\rangle \\
  =& \langle V f(D_1)\rho_1(1-\psi_1)V^*x\ ,\ y\rangle
  \end{array}
  \]
  for every $x,y\in H_2$. It concludes that
  \[
  f(D_2)\rho_2(1-\psi_2)=Vf(D_1)\rho_1(1-\psi_1)V^*\in B(H_2).
  \]
  According to lemma \ref{V}, the following equations hold
  \[
  \begin{array}{lll}

  \Phi[f(D_1)\oplus 0] & = & \Phi[(f(D_1)\rho_1(1-\psi_1))\oplus 0]\\

   & = & [(Vf(D_1)\rho_1(1-\psi_1)V^*)\oplus 0]\\

   & = & [(f(D_2)\rho_2(1-\psi_2))\oplus 0]\\

   & = & [f(D_2)\oplus 0].

  \end{array}
  \]
  \end{proof}

  \begin{lem}\label{nomalizingfunction}
  For every normalizing function $\chi$, there is $\Psi[\chi(D_1)\oplus 0]=[\chi(D_2)\oplus 0]$.
  \end{lem}

  \begin{proof}
  It is known that $\chi(D_i)\in D^*(M_i, H_i)^G$ for $i=1,2$ (cf. \cite{Higson2000}, \cite{Rufus}), which concludes $\chi(D_i)\oplus 0\in D^*(M_i, \mathcal{H}_i)^G$.

  Suppose the normalizing function $\chi$ satisfies its distributional Fourier transform $\widehat{\chi}$ is supported in $(-r,r)$. Denote by $T$ the operator
  \[
  V\chi(D_1)V^*-\chi(D_2).
  \]
  It could be derived that $T\oplus 0$ is pseudolocal operator with propagation less than $r$ from the proof of lemma \ref{cmap} and lemma \ref{dmap}. The only thing we need to do is to prove that $T$ is supported near $Z_2$ and locally compact on $M_2-Z_2$ from which we can get that $T\oplus 0$ is also supported near $Z_2$ and locally compact on $M_2-Z_2$. These can deduce that
  \[
  T\oplus 0\in D^*_{Z_2}(M_2,\mathcal{H}_2)^G,
  \]
  and
  \[
  \Psi[\chi(D_1)\oplus 0]=[(V\chi(D_1)V^*)\oplus 0]=[\chi(D_2)\oplus 0].
  \]

  The operator $T$ is supported near $Z_2$ due to the fact that
  \[
  \rho_2(f)T=0,T\rho_2(f)=0,
  \]
  \[
  \forall f\in C_c(M_2) \text{  and } d(\text{supp}(f),Z_2)>r,
  \]
  which can be proved by similar method in lemma \ref{c0function}.

  For any $f\in C_c(M_2-Z_2)$, there is $r'>0$ such that
  \[
  d(\text{supp}(f),Z_2)>r'.
  \]
  Take a normalizing function $\eta$ satisfying $\widehat \eta$ is supported in $(-r',r')$. Similarly, there are
  \[
  \left(V\eta(D_1)V^*-\eta(D_2)\right)\rho_2(f)=0,
  \]
  \[
  \rho_2(f)\left(V\eta(D_1)V^*-\eta(D_2)\right)=0.
  \]
  Let $\tau=\chi-\eta$ which is in $C_0(\mathbb{R})$. Then
  \[
  \begin{array}{lll}
   T& =&  V\chi(D_1)V^*-\chi(D_2) \\
   & = &V\eta(D_1)V^*-\eta(D_2)+V\tau(D_1)V^*-\tau(D_2).
  \end{array}
  \]
  Consequently,
  \[
  \rho_2(f)T=\rho_2(f)V\tau(D_1)V^*-\rho_2(f)\tau(D_2),
  \]
  \[
  T\rho_2(f)=V\tau(D_1)V^*\rho_2(f)-\tau(D_2)\rho_2(f),
  \]
  are all compact operators since $\tau(D_1)$ and $\tau(D_2)$ are locally compact on $M_2$. Thus $T$ is locally compact on $M_2-Z_2$.

  The lemma has been proved.
  \end{proof}


  Define $C^*$-algebras as follows
  $$\mathcal{A}=\left\{(T_1,T_2)\in C^*(M_1,\mathcal{H}_1)^G\oplus C^*(M_2,\mathcal{H}_2)^G:\Phi[T_1]=[T_2]\right\},
  $$
  $$
  \mathcal{B}=\left\{(S_1,S_2)\in D^*(M_1,\mathcal{H}_1)^G\oplus D^*(M_2,\mathcal{H}_2)^G:\Psi [S_1]=[S_2]\right\},
  $$
  \begin{lem}
  $\mathcal{A}$ is a closed ideal in $\mathcal{B}$.
  \end{lem}

  \begin{proof}
  If $(T_1, T_2)\in \mathcal{A}$, then $(V\otimes 1)T_2(V^*\otimes 1)-T_2$ is in $C^*_{Z_2}(M_2,\mathcal{H}_2)^G$ which is contained in $D^*_{Z_2}(M_2,\mathcal{H}_2)^G$. It concludes that $(T_1, T_2)$ is in $\mathcal{B}$ and $\mathcal{A}$ is a $C^*$-subalgebra of $\mathcal{B}$.

  Take $(T_1, T_2)\in \mathcal{A}$ and $(S_1, S_2)\in \mathcal{B}$. Then
  \[
  {\begin{array}{lll}
  (V\otimes 1)T_1S_1(V^*\otimes 1)-T_2S_2 & = & (V\otimes 1)T_1\chi_{Z_1^c}S_1(V^*\otimes 1) \\
   &
  +&\big((V\otimes 1)T_1(V^*\otimes 1)-T_2\big)(V\otimes 1)S_1(V^*\otimes 1)\\
   &+&T_2\big((V\otimes 1)S_1(V^*\otimes 1)-S_2\big).
  \end{array}}
  \]
  The analysis in lemma \ref{cmap} and lemma \ref{dmap} can show that the three terms on the right side are all in $C^*_{Z_2}(M_2,\mathcal{H}_2)^G$, which means $(T_1, T_2)\cdot(S_1, S_2)$ is in $\mathcal{A}$. Similarly, $(S_1, S_2)\cdot(T_1, T_2)$ is also in $\mathcal{A}$ which can be shown by using *-operation.

  As a result, $\mathcal{A}$ is a closed ideal in $\mathcal{B}$.
  \end{proof}

  The exact sequence of $C^*$-algebras
  \[
  \xymatrix@C=0.5cm{
    0 \ar[r] & \mathcal{A} \ar[rr] && \mathcal{B} \ar[rr] && \mathcal{B}/\mathcal{A} \ar[r] & 0, }
  \]
  induces a boundary map in $K$-theory
  \[
  \partial: K_{m+1}(\mathcal{B}/\mathcal{A})\longrightarrow K_m(\mathcal{A}),
  \]
  where $m=\dim M_1=\dim M_2$.

  By our assumption, the map $Z_2\longrightarrow Z_1$ defined by $h^{-1}$ on $Z_2-Y_2$ and $q$ on $Y_2$ is a $G$-equivariant coarse equivalence. Since $\chi_{Z_i}\mathcal{H}_i=L^2(Z_i, S_i)\otimes l^2(G)$ are locally free geometric $Z_i$-$G$ modules in a natural way, the map $Z_2\longrightarrow Z_1$ has a $G$-equivariant isometric cover $Q:\chi_{Z_2}\mathcal{H}_2\longrightarrow \chi_{Z_1}\mathcal{H}_1$. The existence of this $Q$ can be guaranteed by the construction in Chapter 4 \cite{Rufus}. Define
  \[
  W=\left(\begin{array}{cc}
  Q&  0 \\
  0 &V^*\otimes 1
  \end{array}\right) :\mathcal{H}_2\longrightarrow \mathcal{H}_1,
  \]
  with respect to the decomposition $\mathcal{H}_i=\chi_{Z_i}\mathcal{H}_i\oplus \chi_{Z_i^c}\mathcal{H}_i$ for $i=1,2$, then $W$ is a $G$-equivariant isometric cover for $M_2\longrightarrow M_1$. Furtherly, there is a *-homomorphism
  \[
  C^*(M_2,\mathcal{H}_2)^G\longrightarrow C^*(M_1,\mathcal{H}_1),\ T\longmapsto WTW^*.
  \]
  Using this operator $W$, the exact sequence of $C^*$-algebras
  \[\xymatrix@C=0.5cm{
    0 \ar[r] & C^*_{Z_1}(M_1,\mathcal{H}_1)^G \ar[rr] && \mathcal{A} \ar[rr] && C^*(M_2,\mathcal{H}_2)^G \ar[r] & 0 }
  \]
  \[
  T\mapsto (T,0) \ \ \ \ \ \ \  \ (T,S)\mapsto S
  \]
  has a splitting *-homomorphism
  \[
  \text{ad}_W:C^*(M_2,\mathcal{H}_2)^G\longrightarrow \mathcal{A},\ S\mapsto (WSW^*,S).
  \]
  In fact, for every $S$ in $C^*(M_2,\mathcal{H}_2)^G$, by the decomposition $\mathcal{H}_i=\chi_{Z_i}\mathcal{H}_i\oplus \chi_{Z_i^c}\mathcal{H}_i$, there is
  \[
  (V\otimes 1)WSW^*(V^*\otimes 1)-S
  \]
  \[
  \begin{array}{ll}
  =& \footnotesize{\left(\begin{array}{cc} 0 & 0\\ 0 & V\otimes 1 \end{array}\right)\left(\begin{array}{cc}
  Q&  0 \\
  0 &V^*\otimes 1
  \end{array}\right)S\left(\begin{array}{cc}
  Q^*&  0 \\
  0 &V\otimes 1
  \end{array}\right)\left(\begin{array}{cc} 0 & 0\\ 0 & V^*\otimes 1 \end{array}\right)-S} \\
  = & \chi_{Z_2^c}S\chi_{Z_2^c}-S\\
  = & -\chi_{Z_2}S\chi_{Z_2}-\chi_{Z_2}S\chi_{Z_2^c}-\chi_{Z_2^c}S\chi_{Z_2}\in C_{Z_2}^*(M_2,\mathcal{H}_2)^G,
  \end{array}
  \]
  which means $(WSW^*,S)\in\mathcal{A}$.

  The split exact sequence of $C^*$-algebras
  \[\xymatrix@C=0.5cm{
    0 \ar[r] & C^*_{Z_1}(M_1,\mathcal{H}_1)^G \ar[rr] && \mathcal{A} \ar[rr] && C^*(M_2,\mathcal{H}_2)^G \ar[r] & 0 }
  \]
  concludes that
  \[
  K_*(\mathcal{A})\cong K_*(C^*_{Z_1}(M_1,\mathcal{H}_1)^G)\oplus K_*(C^*(M_2,\mathcal{H}_2)^G).
  \]
  On the other hand, there are explicit isomorphisms
  \[
  K_*(C^*_{Z_1}(M_1,\mathcal{H}_1)^G)\cong K_*(C^*(Z_1,\chi_{Z_1}\mathcal{H}_1)^G)\cong K_*(C^*(Y_1,\chi_{Y_1}\mathcal{H}_1)^G).
  \]
  As a consequence,
  \[
  K_*(\mathcal{A})\cong  K_*(C^*(Y_1)^G)\oplus K_*(C^*(M_2)^G), i=1,2.
  \]

  Let $M_1\coprod M_2$ be the disjoint union of $M_1$ and $M_2$, $S_1\coprod S_2$ be the Dirac bundle on $M$. Denote by $D$ the Dirac operator on $S_1\coprod S_2$ and $D$ equals to $D_i$ on $S_i$.

  When $m$ is odd, lemma \ref{nomalizingfunction} tells us that
  $\chi(D)\oplus 0$ is in $\mathcal{B}$ for every normalizing function $\chi$. Moreover, $\frac{1+\chi(D)}{2}\oplus 0$ is also in $\mathcal{B}$ according to the fact that
  \[
  (V\otimes 1)(1\oplus 0)(V^*\otimes 1)-1\oplus 0=-\chi_{Z_2}\oplus 0\in D_{Z_2}^*(M_2,\mathcal{H}_2)^G.
  \]
  The equivalence class of $\frac{1+\chi(D)}{2}\oplus 0$ in $\mathcal{B}/\mathcal{A}$ is a projective element which is independent of the choice of $\chi$ by lemma \ref{c0function}. 

  When $m$ is even, the bundle $S$ has a $\mathbb{Z}_2$-grading  $S=S'\oplus S''$ and $$D=\left(\begin{array}{ll} 0 & D_+ \\ D_- & 0\end{array}\right)$$ with respect to this grading of $S$. 
  For every normalizing function $\chi$, there is a decomposition
  $$
  \chi(D)=\left(\begin{array}{cc} 0 & \chi(D)_+ \\
  \chi(D)_- & 0
  \end{array}\right)
  $$ with respect to $L^2(M,S')\oplus L^2(M,S'')$. Moreover, $\chi(D)_-\oplus 0\in \mathcal{B}$ and its equivalence class in $\mathcal{B}/\mathcal{A}$ is a unitary element independent of the choice of $\chi$.

  In generally, let
  \[
  [\chi(D)]=\left\{
              \begin{array}{ll}
                \left[\frac{1+\chi(D)}{2}\oplus 0\right]\in K_{m+1}(\mathcal{B}/\mathcal{A}), & \hbox{m is odd;} \\
                \left[\chi(D)_-\oplus 0 \right]\in K_{m+1}(\mathcal{B}/\mathcal{A}), & \hbox{m is even.}
              \end{array}
            \right.
  \]

  The equivariant coarse index of $D$ is defined by
  \[
  \text{Index}^G D=\partial [\chi(D)]\in K_m(\mathcal{A}).
  \]
  Its image under the map
  \[
  K_*(\mathcal{A})\longrightarrow K_*(C^*(M_1,\mathcal{H}_1)^G)\oplus K_*(C^*(M_2,\mathcal{H}_2)^G),
  \]
  which is induced by the natural inclusion
  \[
  \mathcal{A}\longrightarrow C^*(M_1,\mathcal{H}_1)^G\oplus C^*(M_1,\mathcal{H}_1)^G,
  \]
  is exactly the pair of the equivariant coarse index of $D_1$ and $D_2$
  \[
  \left(\text{Index}^G D_1,\text{Index}^G D_2\right)\in K_*(C^*(M_1,\mathcal{H}_1)^G)\oplus K_*(C^*(M_2,\mathcal{H}_2)^G).
  \]
  \begin{defn}\label{defofrelativeindex}
  For a set of relative equivariant coarse index data $(M_i,S_i,D_i,Y_i,h,q)$ over $G$, define the relative equivariant coarse index as the component of $\text{Index}^GD$ in $K_m(C^*(Y_1)^G)$ under the split exact sequence
  \[\xymatrix@C=0.5cm{
    0 \ar[r] & K_m(C^*_{Z_1}(M_1)^G) \ar[rr] && K_m(\mathcal{A}) \ar[rr] && K_m(C^*(M_2)^G) \ar[r] & 0 }
  \]
  and the natural isomorphism
  \[
  K_m(C^*_{Z_1}(M_1)^G)\cong K_m(C^*(Y_1)^G).
  \]
  We denote it by
  \[\text{Index}_r^G(D_1,D_2)\in K_m(C^*(Y_1)^G).\]
  Here $m=\dim M_1=\dim M_2$.
  \end{defn}

  Remark that we omit the symbol $G$ in the upper right corner and leave out the word 'equivariant' everywhere if $G$ is a trivial group.

\subsection{More Expositions on Relative Equivariant Coarse Index}
\label{More Expositions on Relative Equivariant Coarse Index}

\text{ }

  In this subsection we give further explanations about the relative equivariant coarse index in order to show alternative ways to get it and the independence of the intermediate choice in the steps to define it. These properties will increase the flexibility of computing and using our relative equivariant coarse index.

\begin{prop}
  Different choice of the $G$-equivariant isometric cover $$Q:\chi_{Z_2}\mathcal{H}_2\longrightarrow \chi_{Z_1}\mathcal{H}_1$$ for the map $Z_2\longrightarrow Z_1$ does not influence the splitting homomorphism
  \[
  (\text{ad}_W)_*:K_*(C^*(M_2,\mathcal{H}_2)^G)\longrightarrow K_*(\mathcal{A}),
  \]
  and then the decomposition of $K_*(\mathcal{A})$.
  \end{prop}
  \begin{proof}
  Assume that there are $G$-equivariant isometric covers $$Q_i:\chi_{Z_2}\mathcal{H}_2\longrightarrow \chi_{Z_1}\mathcal{H}_1, \ \ i=1,2,$$ for the map $Z_2\longrightarrow Z_1$. Let $W_i:\mathcal{H}_2\longrightarrow \mathcal{H}_1$ be the operators associated to $Q_i$.

  We show that $(W_iW_j^*, 1)$ are multipliers of $\mathcal{A}$ for $i,j\in\{1,2\}$. Obviously, $\mathcal{A}$ is a non-degenerate $C^*$-subalgebra of $B(\mathcal{H}_1\oplus\mathcal{H}_2)$ and $(W_iW_j^*, 1)\in B(\mathcal{H}_1)\oplus B(\mathcal{H}_2)\subseteq B(\mathcal{H}_1\oplus\mathcal{H}_2)$. Take element $(T,S)$ in $\mathcal{A}$, then
  \[
  (W_iW_j^*, 1)(T,S)=(W_iW_j^*T, S)\in C^*(M_1,\mathcal{H}_1)^G\oplus C^*(M_2,\mathcal{H}_2)^G,
  \]
  \[
  (T,S)(W_iW_j^*, 1)=(TW_iW_j^*, S)\in C^*(M_1,\mathcal{H}_1)^G\oplus C^*(M_2,\mathcal{H}_2)^G,
  \]
  since $W_1$ and $W_2$ are $G$-equivariant isometric cover for $M_2\longrightarrow M_1$. By the decomposition \[
  \mathcal{H}_k=\chi_{Z_k}\mathcal{H}_k\oplus \chi_{Z_k^c}\mathcal{H}_k,\  k=1,2,
  \]
  there are
  \[
  (V\otimes 1)W_iW_j^*
  =
  \left(\begin{array}{cc}
  0 & 0\\ 0 & V\otimes 1
  \end{array}\right)
  \left(\begin{array}{cc}
  Q_i & 0\\ 0 & V^*\otimes 1
  \end{array}\right)
  \left(\begin{array}{cc}
  Q_j^* & 0\\ 0 & V\otimes 1
  \end{array}\right)=V\otimes 1,
  \]
  and
  \[
  W_iW_j^*(V^*\otimes 1)
  =
  \left(\begin{array}{cc}
  Q_i & 0\\ 0 & V^*\otimes 1
  \end{array}\right)
  \left(\begin{array}{cc}
  Q_j^* & 0\\ 0 & V\otimes 1
  \end{array}\right)\left(\begin{array}{cc}
  0 & 0\\ 0 & V^*\otimes 1
  \end{array}\right)=V^*\otimes 1,
  \]
  which means
  \[
  (V\otimes 1)W_iW_j^*T(V^*\otimes 1)-S=(V\otimes 1)T(V^*\otimes 1)-S\in C^*_{Z_2}(M_2,\mathcal{H}_2)^G,
  \]
  \[
  (V\otimes 1)T W_iW_j^*(V^*\otimes 1)-S=(V\otimes 1)T(V^*\otimes 1)-S\in C^*_{Z_2}(M_2,\mathcal{H}_2)^G.
  \]
  It concludes that $(W_iW_j^*,1)$ are multipliers of $\mathcal{A}$.

  Define
  \[
  R=\left(\left[\begin{array}{cc}
  1-W_1W_1^* & W_1W_2^*\\
  W_2W_1^* & 1-W_2W_2^*
  \end{array}\right],\left[\begin{array}{cc} 0 & 1\\ 1 & 0\end{array}\right]\right)
  \]
  which is a multiplier of $M_{2\times 2}(\mathcal{A})$ where $M_{2\times 2}(\mathcal{A})$ is the $2\times 2$ matrix algebra of $\mathcal{A}$. The operator $R$ is self-adjoint and $R^2=1$. So the *-isomorphism
  \[
  \text{ad}_R:M_{2\times 2}(\mathcal{A})\longrightarrow M_{2\times 2}(\mathcal{A}), \ T\longmapsto RTR^*,
  \]
  induces the identity map on $K_*(M_{2\times 2}(\mathcal{A}))$. Define the following two *-homomorphisms
  \[
  C^*(M_2,\mathcal{H}_2)^G\longrightarrow M_{2\times 2}(\mathcal{A}),
  \]
  \[
  \alpha: S\longmapsto \left(\left[\begin{array}{cc}
  W_1SW_1^* & 0\\
  0 & 0
  \end{array}\right],\left[\begin{array}{cc} S & 0\\ 0 & 0\end{array}\right]\right),
  \]
  \[
  \beta:S\longmapsto \left(\left[\begin{array}{cc}
  0 & 0\\
  0 & W_2SW_2^*
  \end{array}\right],\left[\begin{array}{cc} 0 & 0\\ 0 & S\end{array}\right]\right),
  \]
  which satisfy $\alpha=\text{ad}_R\circ\beta$. Then
  \[
  \alpha_*=\beta_*: K_*(C^*(M_2,\mathcal{H}_2)^G)\longrightarrow K_*(M_{2\times 2}(\mathcal{A})).
  \]
  As a result,
  \[
  (\text{ad}_{W_1})_*=(\text{ad}_{W_2})_*:K_*(C^*(M_2,\mathcal{H}_2)^G)\longrightarrow K_*(\mathcal{A}).
  \]
  \end{proof}

  \begin{prop}\label{cmap*}
  The following maps
  \[
  \text{ad}_V :C^*(M_1,H_1)^G/C^*_{Z_1}( M_1,H_1)^G\longrightarrow C^*(M_2,H_2)^G/C^*_{Z_2}( M_2,H_2)^G,
  \]
  \[
  [T]\longmapsto [VTV^*],
  \]
  \[
  \text{Ad}_V :D^*(M_1,H_1)^G/D^*_{Z_1}( M_1,H_1)^G\longrightarrow D^*(M_2,H_2)^G/D^*_{Z_2}( M_2,H_2)^G,
  \]
  \[
  [T]\longmapsto [VTV^*],
  \]
  are well-defined *-isomorphisms satisfying
  $$\text{ad}_V[f(D_1)]=[f(D_2)],\ \ \ \ \ \  \text{Ad}_V[\chi(D_1)]=[\chi(D_2)],$$ for every $f\in C_0(\mathbb{R})$ and normalizing function $\chi$.
  \end{prop}
  \begin{proof}
  Most of the statements in the proof of lemma \ref{cmap}, lemma \ref{dmap}, lemma \ref{c0function} and lemma \ref{nomalizingfunction} are still hold when we change $\mathcal{H}_i$ to $H_i$ and $V\otimes 1$ to $V$.
  \end{proof}

  It can be guaranteed by lemma \ref{V} that these diagrams are commutative:
  \[
  \xymatrix{
    C^*(M_1,H_1)^G/C^*_{Z_1}( M_1,H_1)^G \ar[d]_{\oplus 0} \ar[r]^{\text{ad}_V}_{\cong} & C^*(M_2,H_2)^G/C^*_{Z_2}( M_2,H_2)^G \ar[d]^{\oplus 0} \\
    C^*(M_1,\mathcal{H}_1)^G/C^*_{Z_1}( M_1,\mathcal{H}_1)^G \ar[r]^{\Phi}_{\cong} & C^*(M_2,\mathcal{H}_2)^G/C^*_{Z_2}( M_2,\mathcal{H}_2)^G,   }
  \]
  \[
  \xymatrix{
    D^*(M_1,H_1)^G/D^*_{Z_1}( M_1,H_1)^G \ar[d]_{\oplus 0} \ar[r]^{\text{Ad}_V}_{\cong} & D^*(M_2,H_2)^G/D^*_{Z_2}( M_2,H_2)^G \ar[d]^{\oplus 0} \\
    D^*(M_1,\mathcal{H}_1)^G/D^*_{Z_1}( M_1,\mathcal{H}_1)^G \ar[r]^{\Psi}_{\cong} & D^*(M_2,\mathcal{H}_2)^G/D^*_{Z_2}( M_2,\mathcal{H}_2)^G.   }
  \]

  Similarly, there are $C^*$-algebras
  $$A=\left\{(T_1,T_2)\in C^*(M_1,H_1)^G\oplus C^*(M_2,H_2)^G:\text{ad}_V[T_1]=[T_2]\right\},$$
  $$
  B=\left\{(S_1,S_2)\in D^*(M_1,H_1)^G\oplus D^*(M_2,H_2)^G:\text{Ad}_V[S_1]=[S_2]\right\},
  $$
  and $A$ is a closed ideal in $B$. The following two exact sequence are commutative
  \[
  \xymatrix{
    0  \ar[r] & A \ar[d]_{\oplus 0} \ar[r] & B \ar[r] \ar[d]_{\oplus 0}  & B/A \ar[d]_{\oplus 0}\ar[r]  & 0  \\
    0 \ar[r] & \mathcal{A} \ar[r] & \mathcal{B} \ar[r] & \mathcal{B}/\mathcal{A} \ar[r] & 0.   }
  \]
  Then there is a commutative diagram in $K$-theory
  \[
  \xymatrix{
    K_*(B/A) \ar[d]_{\oplus 0} \ar[r]^{\partial'} & K_*(A) \ar[d]^{\oplus 0} \\
    K_*(\mathcal{B}/\mathcal{A}) \ar[r]^{\partial} & K_*(\mathcal{A}).   }
  \]
  Define
  \[
  [\chi(D)]'=\left\{
              \begin{array}{ll}
                \left[\frac{1+\chi(D)}{2}\right]\in K_{m+1}(B/A), & \hbox{m is odd;} \\
                \left[\chi(D)_- \right]\in K_{m+1}(B/A), & \hbox{m is even;}
              \end{array}
            \right.
  \]
  which is independent of the choice of the normalizing function $\chi$. As a result,
  \[
  \text{Index}^GD=\partial([\chi(D)]'\oplus 0)=(\partial'[\chi(D)]')\oplus 0.
  \]

  Moreover, the following two exact sequence are also commutative
  \[
  \xymatrix{
    0  \ar[r] & C^*_{Z_1}(M_1,H_1)^G \ar[d]_{\oplus 0} \ar[r] & A \ar[r] \ar[d]_{\oplus 0}  & C^*(M_2,H_2)^G \ar[d]_{\oplus 0}\ar[r]  & 0  \\
    0 \ar[r] & C^*_{Z_1}(M_1,\mathcal{H}_1)^G \ar[r] & \mathcal{A} \ar[r] & C^*(M_2,\mathcal{H}_2)^G \ar[r] & 0.   }
  \]
  In general, we do not know whether the first row of exact sequence is split. That is why we have to introduce $\mathcal{H}_i$ as geometric modules. However, this is unnecessary if $H_i$ themselves are locally free. In this case, the map $Z_2\longrightarrow Z_1$ defined by $h^{-1}$ and $q$ has a $G$-equivariant isometric cover $L:\chi_{Z_2}H_2\longrightarrow \chi_{Z_1}H_1$. Let
  \[
  N=\left(\begin{array}{cc}L & 0\\ 0 & V^*\end{array}\right): H_2\longrightarrow H_1.
  \]
  which is a $G$-equivariant isometric cover of $M_2\longrightarrow M_1$. Then the *-homomorphism
  \[
  \text{ad}_{N}:C^*(M_2,H_2)^G\longrightarrow A, S\longmapsto (NSN^*,S)
  \]
  is a splitting map for the first row of exact sequence. We can choose $Q=L\otimes 1$ to define $W$, then the splitting maps $\text{ad}_N$ in the first row and $\text{ad}_W$ in the second row are also commutative. So we can define the relative equivariant coarse index as the component of $\partial'[\chi(D)]'$ in $K_m(C^*(Y_i,\chi_{Y_i}H_i)^G)$ under the isomorphisms
  \[
  K_*(C^*_{Z_1}(M_1,H_1)^G)\cong K_*(C^*(Y_1,\chi_{Y_1}H_1)^G).
  \]
  These two definitions coincide under the natural isomorphisms
  \[
  K_*(C^*(Y_1,\chi_{Y_1}H_1)^G)\overset{\cong}{\longrightarrow} K_*(C^*(Y_1,\chi_{Y_1}\mathcal{H}_1)^G).
  \]

  Now, we show the independence of the relative equivariant coarse index about the cutoff function pair $(\omega_1,\omega_2)$ and the distance $s$.

  \begin{prop}\label{independence of cut off function pair}
  If we choose another pair of cutoff function $(\bar\omega_1,\bar\omega_2)$ satisfying $\bar\omega_1=\bar\omega_2\circ h$ on $M_1-Z_1$, then the relative equivariant coarse index $\text{Index}_r^G(D_1,D_2)$ in Definition \ref{defofrelativeindex} will not change.
  \end{prop}
  \begin{proof}
  From above observation, it is enough to verify that the group homomorphism
  \[
  \oplus\ 0:K_*(A)\longrightarrow K_*(\mathcal{A})
  \]
  will not change.

  Denote by $\bar\iota_{i}$ the embedding from $H_i$ to $\mathcal{H}_i$ induced by $\bar\omega_i$ for $i=1,2$. Here we have to distinguish $H_i$, $\iota_i(H_i)$ and $\bar\iota_i(H_i)$ explicitly.

  The isometric *-homomorphisms $\oplus \ 0:B(H_i)\longrightarrow B(\mathcal{H}_i)$ actually are
  \[
  \text{ad}_{\iota_i}:B(H_i)\longrightarrow B(\mathcal{H}_i), T\longmapsto \iota_i T\iota_i^*;
  \]
  \[
  \text{ad}_{\bar\iota_i}:B(H_i)\longrightarrow B(\mathcal{H}_i), T\longmapsto \bar\iota_i T\bar\iota_i^*.
  \]
  Then the *-homomorphisms $\oplus \ 0:A\longrightarrow \mathcal{A}$ derived from $(\omega_1,\omega_2)$ and $(\bar\omega_1,\bar\omega_2)$ are
  \[
  \text{ad}_{(\iota_1,\iota_2)}:A\longrightarrow \mathcal{A},(T,S)\longmapsto (\iota_1T\iota_1^*,\iota_2S\iota_2^*);
  \]
  \[
  \text{ad}_{(\bar\iota_1,\bar\iota_2)}:A\longrightarrow \mathcal{A},(T,S)\longmapsto (\bar\iota_1T\bar\iota_1^*,\bar\iota_2S\bar\iota_2^*).
  \]
  Define
  \[
  J=\left(\left[\begin{array}{cc}1-\iota_1\iota_1^* & \iota_1\bar\iota_1^*\\ \bar\iota_1\iota_1^* & 1-\bar\iota_1\bar\iota_1^*\end{array}\right],\left[\begin{array}{cc}1-\iota_2\iota_2^* & \iota_2\bar\iota_2^*\\ \bar\iota_2\iota_2^* & 1-\bar\iota_2\bar\iota_2^*\end{array}\right]\right),
  \]
  which is in $M_{2\times 2}(B(\mathcal{H}_1))\oplus M_{2\times 2}(B(\mathcal{H}_2))\subseteq M_{2\times 2}(B(\mathcal{H}_1)\oplus B(\mathcal{H}_2))$ and is a multiplier of $M_{2\times 2}(\mathcal{A})$. Moreover, $J$ is unitary and self-adjoint. So the *-homomorphism
  \[
  \text{ad}_{J}:M_{2\times 2}(\mathcal{A})\longrightarrow M_{2\times 2}(\mathcal{A})
  \]
  induces the identity map on $K_*(M_{2\times 2}(\mathcal{A}))$.

  Let
  \[
  A\longrightarrow M_{2\times 2}(\mathcal{A}),
  \]
  \[
  \alpha:(T,S)\longmapsto \left(\left[\begin{array}{cc}\iota_1T\iota_1^*& 0 \\ 0&0\end{array}\right],\left[\begin{array}{cc}\iota_2S\iota_2^* & 0\\ 0 &0\end{array}\right]\right);
  \]
  \[
  \beta:(T,S)\longmapsto \left(\left[\begin{array}{cc}0 & 0\\ 0 &\bar\iota_1T\bar\iota_1^*\end{array}\right],\left[\begin{array}{cc}0 & 0 \\ 0 & \bar\iota_2S\bar\iota_2^*\end{array}\right]\right).
  \]
  Then $\alpha=\text{ad}_{R}\circ\beta $ which means
  \[
  \alpha_*=\beta_*:K_*(A)\longrightarrow K_*(M_{2\times 2}(\mathcal{A})).
  \]
  As a result,
  \[
  \left(\text{ad}_{(\iota_1,\iota_2)}\right)_*=(\text{ad}_{(\bar\iota_1,\bar\iota_2)})_*:K_*(A)\longrightarrow K_*(\mathcal{A}).
  \]
  \end{proof}

  \begin{prop}
  If we choose another $s'>0$ to define $Z_i'=O(Y_i,s')$ for $i=1,2$, then the relative equivariant coarse index $\text{Index}_r^G(D_1,D_2)$ in Definition \ref{defofrelativeindex} will not change.
  \end{prop}

  \begin{proof}
  Without loss of generality, we can suppose $s'>s$. Add an apostrophe to mean the objects associated to $s'$, such as $(\omega_1',\omega_2')$, $(\iota_1',\iota_2')$, $V'$, $\Phi'$, $\Psi'$, $\mathcal{A}'$, $\mathcal{B}'$ and $\text{Index}^G(D)'$. According to Proposition \ref{independence of cut off function pair}, we can take $(\omega_1',\omega_2')=(\omega_1,\omega_2)$ and then $(\iota_1',\iota_2')=(\iota_1,\iota_2)$. There are
  \[
  C^*_{Z_i}(M_i)^G\subseteq C^*_{Z_i'}(M_i)^G,\ \ \ \   D^*_{Z_i}(M_i)^G\subseteq D^*_{Z_i'}(M_i)^G
  \]
  by geometric modules $H_i$ or $\mathcal{H}_i$ for $i=1,2$ and
  $$V'=\chi_{Z_2'^c}V\chi_{Z_1'^c}=V\chi_{Z_1'^c}=\chi_{Z_2'^c}V,$$
  $$V'\otimes 1=\chi_{Z_2'^c}(V\otimes 1)\chi_{Z_1'^c}=(V\otimes 1)\chi_{Z_1'^c}=\chi_{Z_2'^c}(V\otimes 1).$$
  Consequently, $\mathcal{A}\subseteq\mathcal{A}'$, $\mathcal{B}\subseteq \mathcal{B}'$ and the following commutative diagrams hold
  \[
  \xymatrix{
    0  \ar[r] & \mathcal{A} \ar[d] \ar[r] &\mathcal{B} \ar[d] \ar[r] & \mathcal{B}/\mathcal{A} \ar[d] \ar[r] & 0  \\
    0 \ar[r] & \mathcal{A}' \ar[r] & \mathcal{B}'\ar[r] & \mathcal{B}'/\mathcal{A}' \ar[r] & 0 ,  }
  \]
  \[
  \xymatrix{
     0  \ar[r] & C^*_{Z_1}(M_1,\mathcal{H}_1)^G \ar[d] \ar[r] & \mathcal{A} \ar[d] \ar[r] & C^*(M_2,\mathcal{H}_2)^G \ar[d] \ar[r] & 0  \\
    0 \ar[r] & C^*_{Z_1'}(M_1,\mathcal{H}_1)^G \ar[r] & \mathcal{A'} \ar[r] & C^*(M_2,\mathcal{H}_2)^G \ar[r] & 0 ,  }
  \]
  where the first diagram means
  \[
  \text{Index}^G(D)\longmapsto \text{Index}^G(D)' \text{\ \ under\ \ }
  K_*(\mathcal{A})\longrightarrow K_*(\mathcal{A}'),
  \]
  and the splitting maps $\text{ad}_W$ and $\text{ad}_{W'}$ in the second diagram are also commutative if we choose $W'=W$.
  Since
  \[
  \xymatrix{
    C^*_{Z_1}(M_1,\mathcal{H}_1)^G \ar[d]  & C^*(Z_1,\chi_{Z_1}\mathcal{H}_1)^G \ar[d]\ar[l]&C^*(Y_1,\chi_{Y_1}\mathcal{H}_1)^G\ar[l]\ar[d]^{=}  \\
    C^*_{Z_1'}(M_1,\mathcal{H}_1)^G  & C^*(Z_1',\chi_{Z_1'}\mathcal{H}_1)^G\ar[l]&C^*(Y_1,\chi_{Y_1}\mathcal{H}_1)^G,\ar[l]   }
  \]
  which is defined by natural embedding, there are
  \[
  \xymatrix{
     K_*(C^*_{Z_1}(M_1,\mathcal{H}_1)^G) \ar[d]\ar[r]^{\cong}  & K_*(C^*(Z_1,\chi_{Z_1}\mathcal{H}_1)^G) \ar[d]^{\cong}\ar[r]^{\cong} & K_*(C^*(Y_1,\chi_{Y_1}\mathcal{H}_1)^G) \ar[d]^{=}\\
    K_*(C^*_{Z_1'}(M_1,\mathcal{H}_1)^G)  \ar[r]^{\cong}& K_*(C^*(Z_1',\chi_{Z_1'}\mathcal{H}_1)^G) \ar[r]^{\cong} &K_*(C^*(Y_1,\chi_{Y_1}\mathcal{H}_1)^G).   }
  \]
  As a result, $\text{Index}^G(D)$ and $\text{Index}^G(D)'$ induce the same one
  \[
  \text{Index}_r^G(D_1,D_2)\in K_m(C^*(Y_1,\chi_{Y_1}\mathcal{H}_1)^G)
  \]
  by diagram chasing.
  \end{proof}

  \begin{prop}
  If we choose another equivariant coarse map $q'$ from $Y_2$ to $Y_1$ close to $q$, then the relative equivariant coarse index will not change.
  \end{prop}

     At the end of this section, we give an example of a set of relative equivariant coarse index data.

  \begin{ex} Let $M_1$ be the manifold $$\cdots \#T^2\# T^2\#T^2\#\Big(S^1\times [0,\infty)\Big)$$ and $M_2$ be $S^1\times \mathbb{R}$ equipped with the Riemannian metrics induced in $\mathbb{R}^3$ where $\#$ means the connected sum of manifolds. Then $S_i=\coprod_{x\in M_i}Cl(T_x M_i)$ are Dirac bundles on $M_i$ with Dirac operator $D_i$ which are isomorphic to signature operators. Take
  $$Y_1=M_1-\Big(S^1\times (1,\infty)\Big), \ \ \ Y_2=S^1\times (-\infty,1].$$
We can put $M_i$ into $\mathbb{R}^3$ along the $y$-axis such that they are symmetric about the $xOy$-plane and $yOz$-plane and the nontrivial actions of $G=\mathbb{Z}_2$ on $M_i$ are defined by
\[
\mathbb{R}^3\longrightarrow \mathbb{R}^3, (x,y,z)\longmapsto (-x,y,z).
\]
  Then obviously there exist a $G$-equivariant coarse map from $M_2$ to $M_1$ satisfying that it is identity when restricted on $M_2-Y_2$ to $M_1-Y_1$ and it is still a coarse map to $Y_1$ when restricted on $Y_2$. So these form a set of relative equivariant coarse index data.
  \end{ex}
\section{Relative Equivariant Coarse Index Theorem}

In this section, we prove the relative equivariant coarse index theorem which indicates a fomulation connecting the relative equivariant coarse index with the localized equivariant coarse indices.

\begin{thm}If
\begin{enumerate}
  \item $(M_i,S_i,D_i, Y_i,h, q)$ is a set of relative equivariant coarse index data and $m$ is the dimension of $ M_i$;
  \item the curvature operator $\mathcal{R}_i$ on the Dirac bundles $S_i$ are uniformly positive outside $Z_i=\overline{O(Y_i,s)}$ for some $s>0$, that is, there is a constant $\varepsilon>0$ such that
  \[
  \mathcal{R}_{ix}\geq \varepsilon^2I,\ \ \forall x\in M_i-Z_i, i=1,2;
  \]
\end{enumerate}
then
\begin{enumerate}
\item $D_i$ has the localized equivariant coarse index
 \[
 \text{Index}_{Z_i}^G (D_i) \in K_m(C^*_{Z_i}(M)^G)\cong K_m(C^*(Y_i)^G),
 \]
 and denote by $\text{Index}_{Y_i}^G (D_i)$ their images in $K_m(C^*(Y_i)^G)$;
\item the identity
\[
\text{Index}^G_r(D_1,D_2)=\text{Index}_{Y_1}^G D_1-q_*\Big(\text{Index}_{Y_2}^G D_2\Big)
\]
holds in $K_m(C^*(Y_1)^G)$.
\end{enumerate}
\end{thm}

\begin{proof}
The existence of $\text{Index}_{Z_i}^G (D_i)$ for $i=1,2$ follows from the construction in section 2. Carry on using the symbols in previous sections. Define a $C^*$-algebra
\[
\mathcal{J}=C^*_{Z_1}(M_1,\mathcal{H}_1)^G\oplus C^*_{Z_2}(M_2,\mathcal{H}_2)^G\subseteq \mathcal{A}.
\]
We know that $f(D)\oplus 0\in \mathcal{J}$ for every $f\in C_c(-\varepsilon,\varepsilon)$ by lemma 2.3 in \cite{Roe2016}. Take a normalizing function $\chi$ satisfying $\chi^2-1$ is supported in $(-\varepsilon,\varepsilon)$.

 When $m$ is odd, the equivalence class of $\frac{\chi(D)+1}{2}\oplus 0$ in $\mathcal{B/J}$ is a projection independent of such kind of $\chi$.

 When $m$ is even, the equivalence class of $\chi(D)_-\oplus 0$ in $\mathcal{B/J}$ is a unitary independent of such kind of $\chi$.

 Let
 \[
  [\chi(D)]''=\left\{
              \begin{array}{ll}
                \left[\frac{1+\chi(D)}{2}\oplus 0\right]\in K_{m+1}(\mathcal{B/J}), & \hbox{m is odd;} \\
                \left[\chi(D)_-\oplus 0 \right]\in K_{m+1}(\mathcal{B/J}), & \hbox{m is even;}
              \end{array}
            \right.
  \]
The following commutative diagram of two exact sequences of $C^*$-algebras
\[
\xymatrix{
   0  \ar[r] & \mathcal{J} \ar[d] \ar[r] & \mathcal{B} \ar[d]^{=} \ar[r] &\mathcal{ B/J} \ar[d] \ar[r] & 0  \\
  0 \ar[r] & \mathcal{A} \ar[r] &\mathcal{ B} \ar[r] & \mathcal{B/A} \ar[r] & 0   }
\]
induces a commutative diagram in $K$-theory:
\[
\xymatrix{
 K_{m+1}(\mathcal{B/J}) \ar[d] \ar[r]^{\ \ \ \partial''} & K_m(\mathcal{J}) \ar[d]  \\
 K_{m+1}(\mathcal{B/A}) \ar[r]^{\ \ \ \partial} & K_m(\mathcal{A})  }
\]
Define
  \[
  \text{Index}_{\text{loc}}^GD=\partial'' [\chi(D)]''\in K_m(\mathcal{J}),
  \]
  which satisfies
  \[
  \xymatrix{
    [\chi(D)]'' \ar[d] \ar[r]^{\partial''} & \text{Index}_{\text{loc}}^GD \ar[d] \\
    [\chi(D)] \ar[r]^{\partial} & \text{Index}^GD,  }
  \]
  and
  \[\text{Index}_{\text{loc}}^GD=\left(\text{Index}^G_{Z_1}D_1,\text{Index}^G_{Z_2}D_2\right)\in K_m(\mathcal{J}).\]

  There is another commutative diagram of $C^*$-algebras
  \[
  \xymatrix{
   0  \ar[r] & C^*(Y_1)^G\ar[d]^{\cong}_{\delta}\ar[r]^{\alpha\ \ \ \ \ \ \ \ } & C^*(Y_1)^G\oplus C^*(Y_2)^G \ar[d]^{\cong}_{\delta'}\ar[r]^{\ \ \ \ \ \ \ \ \ \beta}& C^*(Y_2)^G\ar[d]^{\cong}\ar[r] & 0   \\
   0  \ar[r] & C^*_{Z_1}(M_1)^G\ar[d]^{=} \ar[r]^{\alpha'} & \mathcal{J} \ar[d]_{\delta''} \ar[r]^{\beta'} & C^*_{Z_2}(M_2)^G\ar[d] \ar[r] & 0  \\
  0 \ar[r] & C^*_{Z_1}(M_1)^G \ar[r]^{\alpha''} & \mathcal{A} \ar[r]^{\beta''} & C^*(M_2)^G \ar[r] & 0,
  }
  \]
  where $\alpha$, $\alpha'$, $\alpha''$ are the embedding maps to the first summands and $\beta$, $\beta'$, $\beta''$ are the projective maps to the second summands. The three rows are all split exact sequences of $C^*$-algebras with splitting maps $\gamma$, $\gamma'$ and $\gamma''$ respectively where $\gamma'$ and $\gamma''$ are exactly $\text{ad}_W$ described in the previous section and $\gamma$ is induced by $\gamma'$ on graph above. After transfering it to $K$-theory level, the commutative diagram becomes
  {\scriptsize
  \[
  \xymatrix{
   0  \ar[r] & K_*(C^*(Y)^G)\ar[d]^{\cong}_{\delta_*}\ar[r]^{\alpha_*\ \ \ \ \ \ \ \ \ \ \ \ \ \ \ \ } & K_*(C^*(Y)^G)\oplus K_*(C^*(Y)^G) \ar[d]^{\cong}_{\delta'_*}\ar[r]^{\ \ \ \ \ \ \ \ \ \ \ \ \ \beta_*}& K_*(C^*(Y)^G)\ar[d]^{\cong}\ar[r] & 0   \\
   0  \ar[r] & K_*(C^*_{Z_1}(M_1)^G) \ar[d]^{=} \ar[r]^{\ \ \ \ \ \ \alpha'_*} & K_*(\mathcal{J}) \ar[d]_{\delta''_*} \ar[r]^{\beta'_*\ \ \ \ \ } & K_*(C^*_{Z_2}(M_2)^G) \ar[d] \ar[r] & 0  \\
  0 \ar[r] & K_*(C^*_{Z_1}(M_1)^G) \ar[r]^{\ \ \ \ \ \ \ \alpha''_*} & K_*(\mathcal{A}) \ar[r]^{\beta''_*\ \ \ \ \ } & K_*(C^*(M_2)^G) \ar[r] & 0 .  }
  \]
  }

  By our definition of the relative equivariant coarse index, component of $\text{Index}^GD$ in $K_m(C^*_{Z_1}(M_1)^G)$ is mapped to $\text{Index}_r^G(D_1,D_2)\in K_m(C^*(Y_1)^G)$ exactly under the isomorphism
  \[
  \delta^{-1}_*:K_m(C^*_{Z_1}(M_1)^G)\overset{\cong}{\longrightarrow}K_m(C^*(Y_1)^G).
  \]
  That is,
  \[
  \text{Index}_r^G(D_1,D_2)=\delta^{-1}_*\circ\alpha''^{-1}_*\circ(\text{id}-\gamma''_*\circ\beta''_*)\left(\text{Index}^G(D)\right).\tag{*}
  \]
  We have known that $\text{Index}_{\text{loc}}^G(D)$ is a preimage of $\text{Index}^G(D)$ under the map
  \[
  \delta''_*:K_m(\mathcal{J})\longrightarrow K_m(\mathcal{A}),
  \]
  and obviously
  \[
  \left(\text{Index}^G_{Y_1}(D_1),\text{Index}^G_{Y_2}(D_2)\right)\in  K_m(C^*(Y_1)^G)\oplus K_m(C^*(Y_2)^G)
  \]
  is a preimage of \[\text{Index}_{\text{loc}}^GD=\left(\text{Index}^G_{Z_1}(D_1),\text{Index}^G_{Z_2}(D_2)\right)\in K_m(\mathcal{J})\]
  under the map
  \[
  \delta'_*:K_*(C^*(Y_1)^G)\oplus K_*(C^*(Y_2)^G) \overset{\cong}{\longrightarrow}K_m(\mathcal{J}).
  \]
  So the relationship (*) can be analyzed by diagram chasing
  \[
  \begin{array}{lll}
  \text{Index}_r^G(D_1,D_2) & = &\delta^{-1}_*\circ\alpha''^{-1}_*\circ\big(\text{id}-\gamma''_*\circ\beta''_*\big)\left(\text{Index}^GD\right)\\
   & = &\delta^{-1}_*\circ\alpha'^{-1}_*\circ\big(\text{id}-\gamma'_*\circ\beta'_*\big)\left(\text{Index}^G_{\text{loc}}D\right)\\
  & = & \alpha^{-1}_*\circ(\text{id}-\gamma_*\circ\beta_*)\left(\text{Index}^G_{Y_1}D_1,\text{Index}^G_{Y_2}D_2\right)\\
  & = &
  \alpha^{-1}_*\left(\text{Index}^G_{Y_1} D_1-q_*(\text{Index}^G_{Y_1} D_2),\ 0\right)\\
  &=&
  \text{Index}^G_{Y_1} D_1-q_*\Big(\text{Index}^G_{Y_2} D_2\Big).
  \end{array}
  \]
  As a result, we get that
  \[
  \text{Index}_r^G(D_1,D_2)=\text{Index}^G_{Y_1}(D_1)-q_*\Big(\text{Index}^G_{Y_2}(D_2)\Big)\in K_m(C^*(Y_1)^G).
  \]

\end{proof}

\section{Relative $L^2$-Index}\label{Relative L2 Index}

In this section, we define the relative $L^2$-index and give a  calculation formula after introducing a positive trace on the localized equivariant Roe algebras.

\subsection{Positive Traces and Canonical Examples}

\text{ }

In this subsection, we present the basic facts about positive traces and recall some canonical examples. More expositions can be found in Section 2.3 of \cite{Rufus}.

Let $C$ be a $C^\ast$-algebra, $C_+$ the collection of all positive elements of $C$ and $C^+$ the unitization of $C$.

\begin{defn}(cf. \cite{Rufus})
    A map $\tau:C_+\to [0,\infty]$ is called a positive trace on $C$ if
    \begin{enumerate}
        \item $\tau(0)=0$;
        \item for all $a\in C$, $\tau(a^\ast a)=\tau(aa^\ast)$;
        \item for all $a_1,a_2\in C_+$ and all $\lambda_1,\lambda_2\geq 0$, $\tau(\lambda_1a_1+\lambda_2a_2)=\lambda_1\tau(a_1)+\lambda_2\tau(a_2)$.
    \end{enumerate}
\end{defn}

For a positive trace $\tau$ on $C$, the set
$$C_\tau=\text{span}\{a\in C_+:\tau(a)<\infty\}=\{a\in C:\tau(|a|)<\infty\}$$
is an algebraic *-ideal in $C$ and it is not necessarily norm-closed. The $\tau$ can extend to a *-preserving linear functional $\tau:C_{\tau}\longrightarrow \mathbb{C}$. Let
\[
\sqrt{C_\tau}=\{a\in C:\tau(a^*a)<\infty\},
\]
which is also an algebraic *-ideal in $C$. There is the fact that
$C_\tau=\sqrt{C_\tau}\cdot\sqrt{C_\tau}$ and $\tau(ab)=\tau(ba)$ if $a,b\in\sqrt{C_\tau}$ or $a\in C_\tau, b\in C$. Moreover, we give a lemma as follows.

\begin{lem}\label{traceformulaextension}
If $\tau$ is a positive trace on a $C^*$-algebra $C$ and $C$ is a closed ideal in a $C^*$-algebra $C'$, then $\sqrt{C_\tau}$ and $C_\tau$ are algebraic ideals in $C'$ and
\[
\tau(ab)=\tau(ba),\ \ \forall a\in C_\tau,b\in C'.
\]
\end{lem}
\begin{proof}
If $a\in C'$ and $b\in \sqrt{C_\tau}$, then $ab\in C$ and $(ab)^*(ab)\leq ||a||^2 b^*b$ which means $ab\in \sqrt{C_\tau}$. So $\sqrt{C_\tau}$ is an algebraic ideal in $C'$.

If $a\in C'$, $b\in C_\tau$ and $b\geq 0$, then $\sqrt{b}\in \sqrt{C_\tau}$ and $ab=a\sqrt{b}\sqrt{b}\in C_\tau$. This concludes that $C_\tau$ is an algebraic ideal in $C'$.

To prove the formula $\tau(ab)=\tau(ba)$ for every $a\in C_\tau$ and $b\in C'$, we can suppose that $a\geq 0$. Then $$\tau(ab)=\tau(\sqrt{a}\sqrt{a}b)=\tau(\sqrt{a}b\sqrt{a})=\tau(b\sqrt{a}\sqrt{a})=\tau(ba).$$
\end{proof}

We call the elements in $C_\tau$ the trace-class elements about $\tau$ and the elements in $\sqrt{C_\tau}$ the Hilbert-Schmidt element about $\tau$.

\begin{defn}(cf. \cite{Rufus})
A positive trace $\tau$ on $C$ is called lower semicontinuous if for any norm convergent sequence $(a_n)$ in $C_+$,
        $$\tau(\lim_{n\to\infty} a_n)\le \liminf_{n\to\infty}\tau(a_n).$$
\end{defn}

\begin{defn}(cf. \cite{Rufus})
    A positive trace $\tau$ on $C$ is called densely defined if the set $\{a\in C_+: \tau(a)<\infty\}$ is dense in $C_+$, or equivalently, $C_\tau$ is dense in $C$.
\end{defn}

\begin{defn}
A continuous linear functional $\tau:C\longrightarrow \mathbb{C}$ on the $C^*$-algebra $C$ is called a bounded positive trace if $\tau (ab)=\tau(ba)$ for every $a,b\in C$ and $\tau(c)\geq 0$ for every $c\in C_+$.
\end{defn}

\begin{rem}
Every bounded positive trace on $C$ can be regarded as a densely defined, lower semicontinuous positive trace.
\end{rem}

If $\tau$ is a positive trace on $C$, then for every positive integer $n\in \mathbb{Z}_+$, the map
\[
\tau^n:M_{n\times n}(C)_+\longrightarrow [0,\infty],
\]
\[
a=(a_{ij})\longmapsto \sum_{j=1}^n \tau(a_{jj}),
\]
is a positive trace on $M_{n\times n}(C)$. If $\tau$ is lower semicontinuous or densely defined, so is $\tau^n$.

Suppose that $\tau$ is a densely defined, lower semicontinuous positive trace on $C$. The Lemma 2.3.13 and Theorem 2.3.16 in \cite{Rufus} tell us that the map
\[
K_*(C_\tau)\longrightarrow K_*(C)
\]
induced by the inclusion is an isomorphism. Therefore, for every element $x\in K_0(C)$, there are idempotents $e$ and $f$ in $M_{n\times n}(C_\tau ^+)$ for some positive integer $n$ such that $x=[e]-[f]$ and $e-f$ in $M_{n\times n}(C_\tau)$. 
Define a group homomorphism
\[
\tau_*:K_0(C)\longrightarrow \mathbb{R}, x=[e]-[f]\longmapsto \sum_{i=1}^{n}\tau(e_{ii}-f_{ii})=\tau^n(e-f).
\]
This definition doesn't rely on the choice of the representative elements $e$ and $f$ as described above in $M_{n\times n}(C_\tau ^+)$ .

Now we look at some examples of canonical traces.

\begin{ex}(cf. \cite{Rufus})
 Let $H$ be a Hilbert space and $(e_j)_{j\in\Lambda}$ a set of orthonormal basis for $H$. The canonical trace on $B(H)$ is defined by
        $$\mathrm{Tr}: B(H)_+\to [0,\infty], \ \  T\longmapsto \sum_{j\in \Lambda}\langle Te_j,e_j\rangle,$$
which is independent of the choice of the orthonormal basis $(e_j)_{j\in\Lambda}$ and is a positive trace on $B(H)$. In this case, $\text{span}\{T\in B(H)_+:\mathrm{Tr}(T)<\infty\}$ is exactly the set of trace-class operators $L^1(H)$. Let $K(H)$ be the set of compact operators on $H$ and restrict $\mathrm{Tr}$ on $K(H)_+$. Then $\mathrm{Tr}$ is a densely defined positive trace on $K(H)$. Call it the canonical trace on $K(H)$. The induced map
\[
\mathrm{Tr}_*:K_0(K(H))\longrightarrow \mathbb{Z},
\]
is a group isomorphism.
\end{ex}

\begin{ex}(cf. \cite{Rufus})
For a countable discrete group $G$, the right regular representation of $G$ on $l^2(G)$ is defined by
\[
\rho:G\longrightarrow U(l^2(G)), g\longmapsto \rho_g,
\]
and
\[
\rho_g(\delta_h)=\delta_{hg^{-1}}, \forall h\in G,
\]
where $\delta_h$ is the characteristic function on the subset $\{h\}$ and obviously $\{\delta_h:h\in G\}$ forms an orthonormal basis for $l^2(G)$. The reduced group $C^*$-algebra of $G$ is given by
\[
C^\ast_\rho(G)=\overline{\text{span}}\{\rho_g:g\in G\}.
\]
The canonical trace on the reduced group $C^\ast$-algebra $C^\ast_\rho(G)$ is defined by
        $$\tau_G: C_\rho^\ast(G)\longrightarrow \mathbb{C}, T\longmapsto \langle T\delta_\mathbb{1}, \delta_\mathbb{1}\rangle,$$
where $\mathbb{1}$ is the identity element of $G$. It is a bounded positive trace on $C_\rho^*(G)$. Then there is a group homomorphism
\[
(\tau_G)_*: K_0\left(C_\rho^*(G)\right)\longrightarrow \mathbb{R}.
\]
\end{ex}

\begin{ex}\label{TrX}(cf. \cite{Rufus})Let $G$ be a countable discrete group acting on a proper metric space $X$ properly and cocompactly by isometries. 
In this case, a bounded fundamental domain always exists according to the Lemma A.2.9 in \cite{Rufus}. Let $(\rho,H, U)$ be a locally free geometric $X$-$G$ module and $E$ a bounded fundamental domain. There is $U_g\chi_E =\chi_{gE}U_g$ for every $g\in G$. Define
\[
J:H\overset{\cong}{\longrightarrow} l^{2}(G)\otimes \chi_E H,\qquad \varphi\longmapsto \sum_{g\in G}\delta_{g}\otimes \chi_{E}U_{g}^{\ast} \varphi
\]
		is a well-defined unitary isomorphism, which is equivariant when $l^2(G)\otimes \chi_E H$ is
		equipped with the tensor product of the left regular representation on $l^2(G)$ and trivial representation on $\chi_E H$. Moreover, regarding $C^{\ast}_{\rho}(G)\otimes  K(\chi_E H)$ as a subset of $B(l^2(G)\otimes \chi_E H)$, conjugation
		by $J$ induces the following isomorphisms
		$$\text{ad}_J:	C^{\ast}(X,H)^{G}\overset{\cong}{\longrightarrow} C^{\ast}_{\rho}(G)\otimes  K(\chi_E H),\ T\longmapsto  JTJ^\ast,$$
where
$$JTJ^\ast=\sum_{g\in G}\rho_g\otimes \chi_E TU_g\chi_E.$$
Regarding
\[C^{\ast}_{\rho}(G)\otimes  K(\chi_E H)\subseteq B(l^2(G))\otimes B(\chi_EH)\subseteq B(l^2(G)\otimes \chi_EH)
\]
and using an orthonormal basis $(e_j)_{j\in \Lambda}$ in $\chi_E H$, 
there is a positive trace $\tau_G\otimes \text{Tr}$ on $C^{\ast}_{\rho}(G)\otimes  K(\chi_E H)$ defined by
\[\tau_G\otimes \text{Tr}:\left(C^{\ast}_{\rho}(G)\otimes  K(\chi_E H)\right)_+\longrightarrow [0,\infty],\]
\[
T\longmapsto \sum_{j\in \Lambda} \langle T(\delta_\mathbb{1}\otimes e_j),(\delta_\mathbb{1}\otimes e_j)\rangle,
\]
which is densely defined, lower seimcontinuous and independent of the choice of the basis involved in its instruction. Then we can define a positive trace on $C^{\ast}(X,H)^{G}$ by
$$\text{Tr}_{X}=(\tau_G\otimes \mathrm{Tr})\circ (\text{ad}_J) :C^*(X,H)^G_+\longrightarrow [0,\infty],$$
$$T\longmapsto (\tau_G\otimes \mathrm{Tr})(JTJ^{\ast})=
	\mathrm{Tr} \left(\chi_{E} T\chi_{{E}}\right).$$
Note that the trace $\text{Trace}_X$ does not depend on the choice of the bounded fundamental domain $E$: for any two choices of $E$, the resulting isomorphisms only differ by conjugation by unitary multipliers of the algebras $C^\ast(X,H)^G$. We call it the canonical trace on $C^*(X,H)^G$. As a result, $\mathrm{Tr}_X$ is a densely defined, lower semicontinuous positive trace. There is a group homomorphism
\[
(\text{Tr}_X)_*: K_0(C^*(X,H)^G)\longrightarrow \mathbb{R},
\]
satisfying the commutative diagram of groups
\[
\xymatrix{
  K_0(C^*(X,H)^G) \ar[d]_{(\text{Tr}_X)_*} \ar[r]^{(\text{ad}_J)_*\ \ \ \ \ \ \ }_{\cong\ \ \ \ \ \ \ } & K_0(C^{\ast}_{\rho}(G)\otimes  K(\chi_E H)) \ar[d]_{(\tau_G\otimes \text{Tr})_*}  & K_0(C_\rho^*(G)) \ar[l]^{\ \ \ \ \ \ \ \ \ \ \ \ \cong}\ar[d]^{(\tau_G)_*} \\
  \mathbb{R} \ar[r]^{=} & \mathbb{R} & \mathbb{R}\ar[l]_{=}   }
\]
where the group isomorphism from $K_0(C_\rho^*(G))$ to $K_0(C^{\ast}_{\rho}(G)\otimes  K(\chi_E H))$ is induced by the *-homomorphism
\[
C_\rho^*(G)\longrightarrow C^{\ast}_{\rho}(G)\otimes  K(\chi_E H), T\longmapsto T\otimes p,
\]
after fixing a one-rank projection $p$ on $\chi_E H$ and it do not rely on the choice of $p$.
\end{ex}

\begin{prop}\label{TrXnaturalproperty}
If a countable discrete group $G$ acting on proper metric spaces $X$ and $X'$ properly, cocompactly by isometries and $f:X\longrightarrow X'$ is a equivariant coarse map, then the following diagram is commutative
\[
\xymatrix{
  K_0(C^*(X,H)^G) \ar[rr]^{f_*} \ar[dr]_{(\text{Tr}_X)_*}
                &  &    K_0(C^*(X',H')^G) \ar[dl]^{\ \ (\text{Tr}_{X'})_*}    \\
                & \mathbb{R}                 },
\]
for locally free geometric modules $H$ and $H'$.
\end{prop}
\begin{proof}
We can construct an equivariant isometric cover $P:H\longrightarrow H'$ for $f$ such that $P(\chi_E H)\subseteq \chi_{E'} H'$ where $E$ and $E'$ are bounded fundamental domains for $X$ and $X'$ respectively \cite{Rufus}. Then for every $T\in C^*(X,H)^G_+$, there is
\[
\text{Tr}_X(T)=\text{Tr}(\chi_ET\chi_E)=\text{Tr}(\chi_{E'}PTP^*\chi_{E'})=\text{Tr}_{X'}(PTP^*).
\]
\end{proof}
\subsection{Positive Traces on Localized Equivariant Roe Algebras}

\text{ }

Based on the previous subsection, we define a positive trace on the localized equivariant Roe algebras and prove some related properties. These results will
be useful to compute the relative $L^2$-index defined in the next subsection and to prove the relative $L^2$-index theorem stated in the last section.

Suppose that a countable discrete group $G$ acting on a proper metric space $X$ properly by isometries and $Z\subseteq X$ is a $G$-invariant closed subset satisfying $Z/G$ is compact. Denote by $Z_{(t)}$ the set $B(Z,t)$ for every positive $t$ then $G$ must act on $Z_{(t)}$ cocompactly. Fix a locally free geometric $X$-$G$ module $(\rho, H,U)$.

Through the natural inclusion, there is
\[
C^*_{Z}( X,  H)^G=\overline{\bigcup_{n=1}^\infty  C^*(Z_{(n)},\chi_{Z_{(n)}} H)^G}.
\]
Since $G$ acts on $Z_{(t)}$ properly, cocompactly and freely by isometries, we have
\[
\text{Tr}_{Z_{(t)}}:C^*(Z_{(t)},\chi_{Z_{(t)}} H)^G_+\longrightarrow [0,\infty],\ \forall t>0
\]
satisfying the commutative diagram
\[
\xymatrix{%
  C^*(Z_{(t)},\chi_{Z_{(t)}} H)^G_+ \ar[rr]^{\text{inclusion}} \ar[dr]_{\text{Tr}_{Z_{(t)}}}
                &  &    C^*(Z_{(r)},\chi_{Z_{(r)}} H)^G_+ \ar[dl]^{\ \ \ \text{Tr}_{Z_{(r)}}}    \\
                & [0,\infty]                }
\]
for every $0<t<r$.

Now define
\[
\mathrm{Tr}_Z^X:C^*_{Z}( X, H)^G_+\longrightarrow [0,\infty],
\]
\[
T\longrightarrow \lim_{t\to \infty}\mathrm{Tr}_{Z_{(t)}}\left(\chi_{Z_{(t)}} T \chi_{Z_{(t)}}\right).
\]
The limitation always exists in $[0,\infty]$ since it is monotonically increasing. There is $\mathrm{Tr}_Z^X=\mathrm{Tr}_{Z_{(t)}}$ on $C^*(Z_{(t)},\chi_{Z_{(t)}} H)^G_+$ for every $t>0$.

\begin{lem}
The map $ \text{Tr}_Z^X$ is a positive trace on $C^*_{ Z}( X,  H)^G$.
\end{lem}

\begin{proof}
Just need to prove that ${\text{Tr}^X_Z}(T^*T)={\text{Tr}^X_Z}(T T^*)$ for every $T\in C^*_{ Z}( M,  H)^G$.

Take a $T\in C^*_{ Z}( M,  H)^G$ and $t>0$. Suppose $E\subseteq {Z_{(t)}}$ is a bounded fundamental domain of $Z_{(t)}$. Decompose $T$ according to $\chi_E H\oplus \chi_{E^c} H$ as
\[
T=\left(\begin{array}{cc} T_1 & T_2 \\ T_3 & T_4\end{array}\right).
\]
Then
\begin{align*}
 \mathrm{Tr}_{
 Z_{(t)}}\left(\chi_{Z_{(t)}} T^*T \chi_{Z_{(t)}}\right)
=\text{Tr}(\chi_ET^*T\chi_E)
=\text{Tr}(T_1^*T_1+T_3^*T_3),\\
\mathrm{Tr}_{Z_{(t)}}\left(\chi_{Z_{(t)}} TT^* \chi_{Z_{(t)}}\right)
=\text{Tr}(\chi_ETT^*\chi_E)
=\text{Tr}(T_1T_1^*+T_3T_3^*),
\end{align*}
and
\[
\text{Tr}(T_1^*T_1+T_3^*T_3)=\text{Tr}(T_1T_1^*+T_3T_3^*).
\]
So
\[
\mathrm{Tr}_{Z_{(t)}}\left(\chi_{Z_{(t)}} T^*T \chi_{Z_{(t)}}\right)=\mathrm{Tr}_{Z_{(t)}}\left(\chi_{Z_{(t)}} TT^* \chi_{Z_{(t)}}\right),
\]
and
\[
{\text{Tr}^X_Z}(T^*T)={\text{Tr}^X_Z}(T T^*).
\]
\end{proof}

\begin{lem}
The positive trace $\mathrm{Tr}_Z^X$ on $C^\ast_{Z}( X,H)^G$ is lower semicontinuous.
\end{lem}
\begin{proof}
Let $(T_n)_{n=1}^\infty$ be a norm convergent sequence in $C^\ast_{Z}(X,H)^G_+$ and the limit is denoted by $T$.

If $\mathrm{Tr}_Z^X(T)<\infty$, then for each $\varepsilon>0$, there exists some $t$ such that
\begin{align*}
\mathrm{Tr}_Z^X(T)
 &\le
\mathrm{Tr}_{Z_{(t)}}(\chi_{Z_{(t)}} T\chi_{Z_{(t)}})+\varepsilon &
 \\
&\le\varliminf_{m\to\infty}\mathrm{Tr}_{ Z_{(t)}}(\chi_{Z_{(t)}} T_m\chi_{Z_{(t)}})+\varepsilon \\
  & \le
\varliminf_{m\to\infty}\mathrm{Tr}_Z^X(T_m)+\varepsilon,
\end{align*}
since $\chi_{Z_{(t)}}T_m\chi_{Z_{(t)}}$ is convergent to $\chi_{Z_{(t)}}T\chi_{Z_{(t)}}$ in the norm topology as $m\rightarrow \infty$ and $\text{Tr}_{Z_{(t)}}$ is lower semicontinuous. As $\varepsilon\to 0^+$, we get the desired inequality.\par

If $\mathrm{Tr}_Z^X(T)=\infty$, then for any $\delta>0$, there exists some $t$ such that
\begin{align*}
\delta
 & \le
\mathrm{Tr}_{Z_{(t)}}(\chi_{Z_{(t)}} T\chi_{Z_{(t)}})\\
 & \le
\varliminf_{m\to\infty}\mathrm{Tr}_{Z_{(t)}}(\chi_{Z_{(t)}} T_m\chi_{Z_{(t)}})\\
 &   \le
\varliminf_{m\to\infty}\mathrm{Tr}_Z^X(T_m),
\end{align*}
As $\delta\to +\infty$, we get the desired result.
\end{proof}
\begin{lem}
The positive trace $\mathrm{Tr}_Z^X$ on $C^\ast_{Z}( X, H)^G$ is densely defined.
\end{lem}
\begin{lem}\label{Lemma:TraceCommuting}The formula
 $$\mathrm{Tr}_Z^X(T_1T_2)=\mathrm{Tr}_Z^X(T_2T_1),$$
 holds if $T_1\in C^\ast_{Z}(X, H)^G$ is a trace-class element about
 $\text{Tr}_Z^X$ and $T_2\in C^*(X,H)^G$.
\end{lem}
\begin{proof}
Since $C^\ast_{Z}(X, H)^G$ is a closed ideal in $C^*(X,H)^G$, this is a consequence of Lemma \ref{traceformulaextension}.
\end{proof}

By now, there is a group homomorphism
\[
\left({\text{Tr}^X_Z}\right)_*:K_0(C^\ast_{Z}( X, H)^G)\longrightarrow \mathbb{R}.
\]

For simplicity, we denote the double of $C^*(X,H)^G$ along $C^*_{ Z}(X,H)^G$ by
\[
I(Z,X,H)^G.
\]
We have known that there is
\[
K_*\left(I(Z,X,H)^G\right)\cong K_*\left(C^*_{ Z}(X,H)^G\right)\oplus K_*\left(C^*(X,H)^G\right).
\]
Extend the group homomorphism
\[
\left({\text{Tr}^X_Z}\right)_*:K_0\left(C^\ast_{Z}( X, H)^G\right)\longrightarrow \mathbb{R},
\]
to $K_0\left(I(Z,X,H)^G\right)$ by zero homomorphism on $K_0\left(C^*(X,H)^G\right)$, then we get
\[
\left( {\text{Tr} }^X_Z\right)_*\oplus 0:K_0\left(I(Z,X,H)^G \right)\longrightarrow \mathbb{R}.
\]
The following lemma gives an explicit formula of $\left( {\text{Tr} }^X_Z\right)_*\oplus 0$ for some special cases.

\begin{lem}\label{Lemma:DoubleTraceComputation}
For every $x\in K_0\left(I(Z,X,H)^G \right)$, there exist idempotents $(e,f)$ and $(e',f')$ in $M_{n\times n}(I(Z,X,H)^{G+})$ for some $n$ satisfying
\[
(e,f)-(e',f')\in M_{n\times n}(I(Z,X,H)^G),
\]
and
\[
x=[(e,f)]-[(e',f')]\in K_0\left(I(Z,X,H)^G \right).
\]
Moreover, if the entries of $e-f$ and $e'-f'$ are trace-class elements about $\text{Tr}_Z^X$, then the following formula holds
\[
\left( {\text{Tr} }^X_Z\right)_*\oplus 0:[(e,f)]-[(e',f')]\longmapsto \sum_{i=1}^n\mathrm{Tr}_Z^X(e_{ii}-f_{ii})-\sum_{i=1}^n\mathrm{Tr}_Z^X(e'_{ii}-f'_{ii}).
\]
\end{lem}

\begin{proof}
	By Lemma \ref{Lemma:TraceCommuting}, we compute that
        \begin{align*}
        &\left( {\text{Tr} }^X_Z\right)_*\oplus 0\left([(e,f)]-[(e',f')]\right)  \\
        =&\left( {\text{Tr} }^X_Z\right)_*\left([\mathcal{E}(e,f)]-[\mathcal{E}(e',f')]\right)\\
        =&\left(\mathrm{Tr}_Z^X\right)^{4n}(\mathcal{E}(e,f)-\mathcal{E}(e',f'))\\
        =&\left(\mathrm{Tr}_Z^X\right)^n(f(e-f)f-f'(e'-f')f')\\
        +&\left(\mathrm{Tr}_Z^X\right)^n((1-f)(e-f)(1-f)-(1-f')(e'-f')(1-f')) \\
        =&\left(\mathrm{Tr}_Z^X\right)^n(f(e-f)f)-\left(\mathrm{Tr}_Z^X\right)^n(f'(e'-f')f')\\
        +&\left(\mathrm{Tr}_Z^X\right)^n((1-f)(e-f)(1-f))-\left(\mathrm{Tr}_Z^X\right)^n((1-f')(e'-f')(1-f')) \\
        =&\left(\mathrm{Tr}_Z^X\right)^n((e-f)f)-\left(\mathrm{Tr}_Z^X\right)^n((e'-f')f')\\
        +&\left(\mathrm{Tr}_Z^X\right)^n((e-f)(1-f))-\left(\mathrm{Tr}_Z^X\right)^n((e'-f')(1-f'))\\
        =&\left(\mathrm{Tr}_Z^X\right)^n(e-f)-\left(\mathrm{Tr}_Z^X\right)^n(e'-f').
        \end{align*}
    \end{proof}

\begin{prop}
If a countable discrete group $G$ acting on proper metric spaces $X$ and $X'$ properly by isometries, $Z\subseteq X$ and $Z'\subseteq X'$ are cocompact $G$-invariant closed subsets, $f:X\longrightarrow X'$ is an equivariant coarse map satisfying $f(Z)\subseteq Z'$, then the following diagram is commutative
\[
\xymatrix{
  K_0(C^*_{Z}(X,H)^G) \ar[rr]^{f_*} \ar[dr]_{(\text{Tr}_Z^X)_*}
                &  &    K_0(C^*_{Z'}(X',H')^G) \ar[dl]^{\ \ (\text{Tr}_{Z'}^{X'})_*}    \\
                & \mathbb{R}                 },
\]
for locally free geometric modules $H$ and $H'$.
\end{prop}
\begin{proof}
By Proposition \ref{TrXnaturalproperty} and the properties of inductive limits.
\end{proof}

In order to define relative $L^2$-index in the next subsection, we introduce some basic facts about the regular $G$-cover.

Let $M$ be a complete Riemannian manifold equipped with a Dirac bundle $S$ and a Dirac operator $D$. A covering map $\pi:M^\#\longrightarrow M$ is a regular cover if $\pi_*(\Pi_1(M^\#))$ is a normal subgroup of $\Pi_1(M)$ where $\Pi_1(M^\#)$ and $\Pi_1(M)$ are the fundamental group of $M^\#$ and $M$ respectively. In this case, $M^\#$ also forms a complete Riemannian manifold if equipped with the pullback structures, such as the Riemannian metric, the Dirac bundle $S^\#$, and the Dirac operator $D^\#$. If a countable discrete group $G$ acts on $M^\#$ properly and freely by isometries and 
$M^\#/G$ is homeomorphic to $M$, then $\pi$ is called a regular $G$-cover.

Assume that $\pi: M^\#\longrightarrow M$ is a regular $G$-cover. Then the trivial group $\{\mathbb{1}\}$ acts on $M$ properly and freely by isometries. Take $H^\#=L^2(M^\#,S^\#)$ and $H=L^2(M,S)$ as the geometric modules which must be locally free. Denote by $E^\#$ the set $\pi^{-1}(E)$ for every $E\subseteq M$. Let $Z$ be a compact subset of $M$ then there is $$\pi^{-1}(B(Z,t))=B(\pi^{-1}(Z),t),\ \ \forall t>0,$$
which is written as $Z_{(t)}^\#$. Since $G$ acts on $Z_{(t)}^\#$ cocompactly and $\{\mathbb{1}\}$ acts on $Z_{(t)}$ cocompactly for every $t>0$, there are densely defined, lower semicontinuous positive trace $\text{Tr}_{Z^\#}^{M^\#}$ on $C^*_{ Z^\#}(M^\#, H^\#)^G$ and $\text{Tr}_{Z}^{M}$ on $C^*_{ Z}(M, H)$.

\subsection{Definition of Relative $L^2$-Index and its Computation}

\text{ }

 In this subsection, we define the relative $L^2$-index and give a calculation formula on certain conditions.

 \begin{defn}
 For a set of relative equivariant coarse index data $(M_i,S_i,D_i,Y_i,h,q)$ over $G$ where $M_i$ are even dimensional and $G$ acts on $Y_i$ cocompactly,
 the relative $L^2$-index is defined by
 \[
 L^2\text{-Index}_r^G(D_1,D_2)=\left(\text{Tr}_{Y_1}\right)_*\left(\text{Index}_r^G(D_1,D_2)^G\right)\in \mathbb{R}.
 \]
 \end{defn}

Now we investigate how to compute the relative $L^2$-index under the conditions that $G$ acts on $M_i$ freely. Recall that we fix a positive real number $s$ and define $Z_i=B(Y_i,s)$ to get the relative equivariant coarse index. Let
$$h_s: M_1-O(Y_1,10s)\longrightarrow M_2-O(Y_2,10s),$$
and define $X=M_1\bigcup_{h_s}M_2$ which is the quotient space of $M_1\coprod M_2$ by identifying $x$ with $h_s(x)$ for every $x\in M_1-O(Y_1,10s)$. Then $M_i$ are homeomorphic to their image $X$ by the natural inclusion $M_i\longrightarrow X$ and we regard $M_i$ are subspaces of $X$. Similarly, define a bundle $S_X$ on $X$ by identifying $v\in S_{1x}$ and $w\in S_{2y}$ if
\[
h_{s}(x)=y, x\in M_1-O(Y_1,10s), y\in M_2-O(Y_210s), \bar{h}(v)=w.
\]

For $x\in O(Y_1,10s)$ and $y\in O(Y_2,10s)$, define
\[
d(x,y)=\inf\{d(x,z)+d(z,y):z\in M_1\cap M_2\}.
\]
The fact that $M_1\cap M_2$ is closed set guarantees that above definition gives a compatible proper metric on $X$ and $M_i\longrightarrow X$ are isometric. Moreover, $G$ acts on $X$ properly and freely by isometries.

Then $H=L^2(X,S_X)$ is a locally free geometric $X$-$G$ module and the natural inclusion $H_i=L^2(M_i,S_i)\longrightarrow H$ is an equivariant isometric 0-cover for $M_i\longrightarrow X$.

Let $Z=Z_1\bigcup Z_2\subseteq X$. Using the *-homomorphisms
\[
 C^*_{Z_1}(M_1,H_1)^G\longrightarrow C^*_{Z}(X,H)^G, T_1\longrightarrow T_1,
\]
\[
 A\longrightarrow I(Z,X,H)^G, (T_1,T_2)\longrightarrow (T_1,NT_2N^*),
\]
\[
C^*(M_2,H_2)^G\longrightarrow C^*(X,H)^G,  T_2\longrightarrow NT_2N^*,
\]
where $N$ is defined in Section \ref{More Expositions on Relative Equivariant Coarse Index}, there is a commutative diagram
\[
\xymatrix{
   0  \ar[r] & C^*_{Z_1}(M_1,H_1)^G \ar[d] \ar[r] & A \ar[d] \ar[r] & C^*(M_2,H_2)^G \ar[d] \ar[r]  &0 \\
  0 \ar[r] & C^*_{Z}(X,H)^G \ar[r] & I(Z,X,H)^G \ar[r] & C^*(X,H)^G \ar[r] & 0 ,  }
\]
which concludes that
\[
\xymatrix{
   0  \ar[r] & K_0(C^*_{Z_1}(M_1)^G) \ar[d] \ar[r] & K_0(A) \ar[d]^{\Delta} \ar[r] & K_0(C^*(M_2)^G) \ar[d] \ar[r]  &0 \\
  0 \ar[r] & K_0(C^*_{Z}(X)^G) \ar[r] & K_0(I(Z,X,H)^G) \ar[r] & K_0(C^*(X)^G) \ar[r] & 0.   }
\]
Furtherly, the splitting maps of above two diagrams are also commutative. On the other hand, since the following commutative diagram holds
\[
\xymatrix{
  K_0(C^*(Y_1)^G) \ar[d]_{(\text{Tr}_{Y_1})_*} \ar[r]^{\cong} & K_0(C^*_{Z_1}(M_1)^G) \ar[d]_{(\text{Tr}_{Z_1}^{M_1})_*}\ar[r]  & K_0(C^*_Z(X)^G)\ar[d]^{(\text{Tr}_Z^X)_*} \\
  \mathbb{R} \ar[r]^{=} & \mathbb{R} \ar[r]^{=}& \mathbb{R} ,  }
\]
we have
\[
 L^2\text{-Index}_r^G(D_1,D_2)=\Big(\left( {\text{Tr} }^X_Z\right)_*\oplus 0\Big)\circ\Delta(\text{Index}^GD).
\]

Take a normalizing function $\chi$ such that $\text{prop}( \chi(D_i))<s$ and $\chi(D_i)^2-1$ are locally trace-class operators in $B(H_i)^G$. This $\chi$ can be found according to the Corollary 9.6.13 in \cite{Rufus}. Then $\text{Index}^GD\in K_0(A)$ equals to
\[
\left[\left(
\begin{array}{cc}
1-(1-ab)^2 & a(2-ba)(1-ba)\\
(1-ba)b & (1-ba)^2
\end{array}
\right),
\left(
\begin{array}{cc}
1-(1-cd)^2 & c(2-dc)(1-dc)\\
(1-dc)d & (1-dc)^2
\end{array}
\right)\right]
\]
\[
-
\left[\left(
\begin{array}{cc}
1 & 0\\
0 & 0
\end{array}
\right),
\left(
\begin{array}{cc}
1 & 0\\
0 & 0
\end{array}
\right)\right],
\]
where $a=\chi(D_1)_-$, $b=\chi(D_1)_+$, $c=\chi(D_2)_-$ and $d=\chi(D_2)_+$. Consequently, $\Delta(\text{Index}^GD)$ is exactly

{\footnotesize
\[
\left[\left(
\begin{array}{cc}
1-(1-ab)^2 & a(2-ba)(1-ba)\\
(1-ba)b & (1-ba)^2
\end{array}
\right),
\left(
\begin{array}{cc}
1-N(1-cd)^2N^* & Nc(2-dc)(1-dc)N^*\\
N(1-dc)dN^* & N(1-dc)^2N^*
\end{array}
\right)\right]
\]
\[
-
\left[\left(
\begin{array}{cc}
1 & 0\\
0 & 0
\end{array}
\right),
\left(
\begin{array}{cc}
1 & 0\\
0 & 0
\end{array}
\right)\right].
\]
}
As a result, by Lemma \ref{Lemma:DoubleTraceComputation} and combined with the fact that
\[
\left(-(1-ab)^2+N(1-cd)^2N^*\right)\chi_{Z_{1(5s)}^c}=0,
\]
\[
\left((1-ba)^2-N(1-dc)^2N^*\right)\chi_{Z_{1(5s)}^c}=0,
\]
we can get a formula for the relative $L^2$-index:
\begin{align*}
&L^2\text{-Index}_r^G(D_1,D_2) &\\
=&\text{Tr}_Z^X\Big(-(1-ab)^2+N(1-cd)^2N^*\Big)+\text{Tr}_Z^X\Big((1-ba)^2-N(1-dc)^2N^*\Big)&\\
=&\text{Tr}_{Z_{1(5s)}}\Big(\chi_{Z_{1(5s)}}(1-ba)^2\chi_{Z_{1(5s)}}\Big)
+\text{Tr}_{Z_{1(5s)}}\Big(\chi_{Z_{1(5s)}}N(1-cd)^2N^*\chi_{Z_{1(5s)}}\Big)&\\
-&\text{Tr}_{Z_{1(5s)}}\Big(\chi_{Z_{1(5s)}}(1-ab)^2\chi_{Z_{1(5s)}}\Big)
-\text{Tr}_{Z_{1(5s)}}\Big(\chi_{Z_{1(5s)}}N(1-dc)^2N^*\chi_{Z_{1(5s)}}\Big)
&\\
=&\text{Tr}_{Z_{1(5s)}}\Big(\chi_{Z_{1(5s)}}(1-ba)^2\chi_{Z_{1(5s)}}\Big)
+\text{Tr}_{Z_{2(5s)}}\Big(\chi_{Z_{2(5s)}}(1-cd)^2\chi_{Z_{2(5s)}}\Big)&\\
-&\text{Tr}_{Z_{1(5s)}}\Big(\chi_{Z_{1(5s)}}(1-ab)^2\chi_{Z_{1(5s)}}\Big)
-\text{Tr}_{Z_{2(5s)}}\Big(\chi_{Z_{2(5s)}}(1-dc)^2\chi_{Z_{2(5s)}}\Big).
&
\end{align*}

\section{Relative $L^2$-Index Theorem}\label{The Relative L2 Index Theorem}
	In the section, we prove the relative $L^{2}$-index theorem, which can be viewed as a relative version of Atiyah's $L^{2}$-index theorem in \cite{Atiyah1992}. It is stated as follows.

\begin{thm}Let $(M_i,S_i,D_i,Y_i,h,q)$ be a set of relative coarse index data with even dimensions where $Y_i$ are compact subsets and $(M_i^\#,S_i^\#,D_i^\#,Y_i^\#,h^\#,q^\#)$ be a set of relative equivariant coarse index data over $G$ acting on $M_i^\#$ freely and on $Y_i^\#$ cocompactly. If the maps $\pi_i:M_i^\#\longrightarrow M_i$ are regular $G$-covers and $M_i^\#$ are equipped with the pullback Riemannian metric, Dirac bundles $S^\#$ and Dirac operators $D_i^\#$ where the subsets $Y_i^\#$ are the preimage of $Y_i$,
 then
\[%
L^2\text{-Index}^G_r\left(D_1^\#,D_2^\#\right)=L^2\text{-Index}_r(D_1,D_2)\in \mathbb{Z}.
\]
\end{thm}

\begin{proof}
Since $Y_1$ is compact, there is
\[
\text{Tr}_*=(\text{Tr}_{Y_1})_* :K_0(C^*(Y_1))\longrightarrow\mathbb{Z}.
\]
So $L^2\text{-Index}_r(D_1,D_2)\in \mathbb{Z}$. Just need to prove that
\[
L^2\text{-Index}^G_r\left(D_1^\#,D_2^\#\right)=L^2\text{-Index}_r(D_1,D_2).
\]

As $Z_{i(5)}$ are compact subsets of $M_i$, there is a $\varepsilon>0$ such that
$$\pi_i:O(x,\varepsilon)\longrightarrow O(\pi(x),\varepsilon)$$
are homeomorphisms for every $x\in Z^\#_{i(5)}$ and $i=1$, $2$.

Following the steps described in the previous subsection, fix a positive number $0<s<\min (1,\frac{\varepsilon}{ 10})$ to define $Z_i=B(Y_i,s)$ and take a normalizing function $\chi$ such that $\text{supp}(\hat \chi)\subseteq (-s,s)$, $\chi(D_i)^2-1$ are locally trace-class operators on $B(H_i)^G$ and $\chi(D_i^\#)^2-1$ are locally trace-class operators on $B(H_i^\#)^G$ where $H_i=L^2(M_i,S_i)$ and $H_i^\#=L^2(M_i^\#,S_i^\#)$.

Suppose that
\[
\chi(D_1)=\left(\begin{array}{cc}0 & b \\ a & 0\end{array}\right),\ \ \
\chi(D_2)=\left(\begin{array}{cc}0 & d \\ c & 0\end{array}\right),
\]
\[
\chi(D_1^\#)=\left(\begin{array}{cc}0 & \bar b \\ \bar a & 0\end{array}\right),\ \ \
\chi(D_2^\#)=\left(\begin{array}{cc}0 & \bar d \\ \bar c & 0\end{array}\right).
\]
We have the formulas that
\begin{align*}
&L^2\text{-Index}_r(D_1,D_2) &\\
=&\text{Tr}\Big(\chi_{Z_{1(5s)}}(1-ba)^2\chi_{Z_{1(5s)}}\Big)
+\text{Tr}\Big(\chi_{Z_{2(5s)}}(1-cd)^2\chi_{Z_{2(5s)}}\Big)&\\
-&\text{Tr}\Big(\chi_{Z_{1(5s)}}(1-ab)^2\chi_{Z_{1(5s)}}\Big)
-\text{Tr}\Big(\chi_{Z_{2(5s)}}(1-dc)^2\chi_{Z_{2(5s)}}\Big);
&
\end{align*}
and
\begin{align*}
&L^2\text{-Index}_r^G(D_1^\#,D_2^\#) &\\
=&\text{Tr}_{Z_{1(5s)}^\#}\Big(\chi_{Z^\#_{1(5s)}}(1-\bar b \bar a)^2\chi_{Z^\#_{1(5s)}}\Big)
+\text{Tr}_{Z^\#_{2(5s)}}\Big(\chi_{Z^\#_{2(5s)}}(1-\bar c \bar d)^2\chi_{Z^\#_{2(5s)}}\Big)&\\
-&\text{Tr}_{Z^\#_{1(5s)}}\Big(\chi_{Z^\#_{1(5s)}}(1-\bar a\bar b)^2\chi_{Z^\#_{1(5s)}}\Big)
-\text{Tr}_{Z^\#_{2(5s)}}\Big(\chi_{Z^\#_{2(5s)}}(1-\bar d\bar c)^2\chi_{Z^\#_{2(5s)}}\Big).
&
\end{align*}
It is possible to find bounded fundamental domains $E^{i}$ for $Z^\#_{i(5s)}$ such that
\[
E^{i}=\coprod_{j=1}^\infty E^i_j, \ \ Z_{i(5s)}=\coprod_{j=1}^\infty F^i_j,
\]
where $E^i_j$ and $F^i_j$ are Borel subsets with diameters less than $s$ and
\[
\pi_i:E^i_j\longrightarrow F^i_j
\]
are homeomorphisms for every $i$ and $j$. As a result,
\begin{align*}
&L^2\text{-Index}_r^G(D_1^\#,D_2^\#) &\\
=&\text{Tr}\Big(\chi_{E^1}(1-\bar b \bar a)^2\chi_{E^1}\Big)
+\text{Tr}\Big(\chi_{E^2}(1-\bar c \bar d)^2\chi_{E^2}\Big)&\\
-&\text{Tr}\Big(\chi_{E^1}(1-\bar a\bar b)^2\chi_{E^1}\Big)
-\text{Tr}\Big(\chi_{E^2}(1-\bar d\bar c)^2\chi_{E^2}\Big)
&\\
=& \sum_{j=1}^\infty  \text{Tr}\Big(\chi_{E^1_j}(1-\bar b \bar a)^2\chi_{E^1_j}\Big)
+\sum_{j=1}^\infty\text{Tr}\Big(\chi_{E^2_j}(1-\bar c \bar d)^2\chi_{E^2_j}\Big)&\\
-&\sum_{j=1}^\infty\text{Tr}\Big(\chi_{E^1_j}(1-\bar a\bar b)^2\chi_{E^1_j}\Big)
-\sum_{j=1}^\infty\text{Tr}\Big(\chi_{E^2_j}(1-\bar d\bar c)^2\chi_{E^2_j}\Big)
&\\
=& \sum_{j=1}^\infty  \text{Tr}\Big(\chi_{F^1_j}(1- b  a)^2\chi_{F^1_j}\Big)
+\sum_{j=1}^\infty\text{Tr}\Big(\chi_{F^2_j}(1- c  d)^2\chi_{F^2_j}\Big)&\\
-&\sum_{j=1}^\infty\text{Tr}\Big(\chi_{F^1_j}(1- a b)^2\chi_{F^1_j}\Big)
-\sum_{j=1}^\infty\text{Tr}\Big(\chi_{F^2_j}(1- d c)^2\chi_{F^2_j}\Big)
&\\
=&L^2\text{-Index}_r(D_1,D_2) .&
\end{align*}

\end{proof}

\noindent{\bf Acknowledgement}\ \
We wish to thank Yi-Jun Yao and Shengzhi Xu for many stimulating discussions and helpful comments. This work is supported in part by the National Science Foundation of China No. 11771143 \& 11801178.




\end{document}